\documentclass{article}
\usepackage[a4paper, portrait, margin=1.1811in]{geometry}
\usepackage[english]{babel}
\usepackage[utf8]{inputenc}
\usepackage[T1]{fontenc}
\usepackage{amsfonts, amsmath, amsthm, amssymb, bbm, mathrsfs, esint, dsfont, mathenv, relsize, stmaryrd}
\usepackage{mathtools}
\usepackage[normalem]{ulem}
\usepackage{enumerate, enumitem,multicol}
\usepackage{algorithm}
\usepackage{algpseudocode}
\usepackage{helvet}
\usepackage{etoolbox}
\usepackage{graphicx}
\usepackage{titlesec}
\usepackage{caption}
\usepackage{booktabs}
\usepackage[dvipsnames]{xcolor} 
\usepackage{csquotes}

\usepackage{setspace}
\graphicspath{ {./images/} }
\usepackage{scrextend}
\usepackage{fancyhdr}
\usepackage{graphicx}
\usepackage{subcaption}
\usepackage[backend=bibtex, natbib=true, style=authoryear, giveninits=true, maxbibnames=99, doi=false,isbn=false,url=false,eprint=false,uniquename=init]{biblatex}
\RequirePackage[colorlinks=true,citecolor=blue!50!black,linkcolor=blue!50!black,urlcolor=blue!50!black,pdfauthor=author]{hyperref}
\usepackage[capitalize,nameinlink,noabbrev]{cleveref}

\theoremstyle{plain}
\newtheorem{theorem}{Theorem}[section] % le [chapter] sert à numéroter par chapitre 
\newtheorem{proposition}[theorem]{Proposition}

\newtheorem{corollary}[theorem]{Corollary}
\newtheorem{lemma}[theorem]{Lemma}
\theoremstyle{remark}

\crefname{thm}{Theorem}{Theorems}
\crefname{lem}{Lemma}{Lemmas}
\crefname{prop}{Proposition}{Propositions}
\crefname{cor}{Corollary}{Corollaries}

\def\EE{\mathbb{E}}

\def\NN{\mathbb{N}}

\def\PP{\mathbb{P}}

\def\RR{\mathbb{R}}

\def\1{\mathds{1}}
\def\A{\mathcal{A}}%%%la lettre A en ronde majuscule
\def\B{\mathcal{B}}%%%la lettre B en ronde majuscule
\def\C{\mathcal{C}}%%%la lettre C en ronde majuscule
\def\G{\mathcal{G}}%%%la lettre G en ronde majuscule
\def\N{\mathcal{N}}%%%la lettre N en ronde majuscule
\newcommand{\Tr}{\mbox{$\mathrm{Tr}$}}
\newcommand{\diag}{\mbox{$\mathrm{diag}$}}
\newcommand{\Cond}{\mbox{$\mathrm{Cond}$}}
\def\grad{\nabla} 
\DeclareMathOperator*{\argmin}{argmin}

\renewcommand{\hat}{\widehat}
\renewcommand{\tilde}{\widetilde}
\renewcommand{\bar}{\overline }

\newlist{hypDiamond}{enumerate}{1} % Avec des Diamants
\setlist[hypDiamond,1]{label={$\diamond$}}

\makeatletter
\newcommand*{\rom}[1]{\mbox{\expandafter\@slowromancap\romannumeral #1@}}
\makeatother
\def\I{\rom{1}}
\def\II{\rom{2}}
\def\III{\rom{3}}

\makeatletter   
\patchcmd{\@maketitle}{\LARGE \@title}{\fontsize{16}{19.2}\selectfont\@title}{}{}
\makeatother

\makeatletter
\def\namedlabel#1#2{\begingroup
   \def\@currentlabel{#2}%
   \label{#1}\endgroup
}
\makeatother

\usepackage{authblk}

\newsavebox\affbox
\author[1]{\textbf{Luca Castelli}}
\author[2]{\textbf{Irène Gannaz}}
\author[1]{\textbf{Clément Marteau}}

\affil[1]{Univ Lyon, Université Claude Bernard Lyon 1, CNRS UMR 5208, Institut Camille Jordan, F-69622 Villeurbanne, France\protect\\~ } 
\affil[2]{Univ. Grenoble Alpes, CNRS, Grenoble INP\footnote{Institute of Engineering Univ. Grenoble Alpes}, G-SCOP, 38000 Grenoble, France}

\titleformat{\section}{\normalfont\fontsize{10}{15}\bfseries}{\thesection.}{1em}{}
\titleformat{\subsection}{\normalfont\fontsize{10}{15}\bfseries}{\thesubsection.}{1em}{}
\titleformat{\subsubsection}{\normalfont\fontsize{10}{15}\bfseries}{\thesubsubsection.}{1em}{}

\titleformat{\author}{\normalfont\fontsize{10}{15}\bfseries}{\thesection}{1em}{}

\title{\textbf{\huge A non-asymptotic upper bound in prediction for the PLS estimator}\\}
\date{\today}

\allowdisplaybreaks[1] % autorise les sauts de page dans des equations

\begin{document}

\pagestyle{headings}	
\newpage
\setcounter{page}{1}
\renewcommand{\thepage}{\arabic{page}}

\setlength{\parskip}{12pt} % saut après paragraphe
\setlength{\parindent}{0pt} % identation en début de paragraphe
\onehalfspacing  % \setstretch{1.3} % espace entre les lignes
	
\maketitle
	
\noindent\rule{15cm}{0.5pt}
	\begin{abstract}
	We investigate the theoretical performances of the Partial Least Square (PLS) algorithm in a high dimensional context. We provide upper bounds on the risk in prediction for the statistical linear model $Y=X\beta+\varepsilon$ when considering the PLS estimator. Our bounds are non-asymptotic and are expressed in terms of the number of observations, the noise level, the properties of the design matrix, and the number of considered PLS components. In particular, we exhibit some scenarios where the variability of the PLS may explode and prove that we can get round of these situations by introducing a Ridge regularization step. These theoretical findings are illustrated by some numerical simulations. 
	\end{abstract}
\noindent\rule{15cm}{0.4pt}

\textbf{\textit{Keywords}}: \textit{Partial least squares; dimension reduction; regression; Ridge regularization}
%%%%%%%%%%%%%%%%%%%%%%%%%%%%%%%%%%%%%%%%%%%%%%%%%%%%%%%%%%%%%%%%%%%%%%%%%%%%%%%%%%%%%%%%%%%%%%

\section{Introduction}

 We observe a $n$-sample $(X_i,Y_i)$, $i=1,\dots,n$, where the $Y_i\in\RR$ are outcome variables of interest and the $X_i\in\RR^p$ $p$-dimensional covariates. We consider a linear relationship within each couple $(X_i,Y_i)$, represented by the equation
\begin{equation}
\label{Eq:Modele lineaire}
Y=X\beta +\varepsilon,
\end{equation}
where $\varepsilon=(\varepsilon_1,\dots \varepsilon_n)^T\sim\N_n\big(0,\tau^2I_n\big)$, $X=(X_1,\dots, X_n)^T\in \RR^{n\times p}$ and $Y=(Y_1,\dots,Y_n)^T$ $\in \RR^n$. Here and below, the matrix $I_n$ is the identity matrix of size $n$, the parameter $\tau >0$ characterizes the noise level and the exponent $T$ denotes the transpose operator.

The linear model \eqref{Eq:Modele lineaire} has been widely investigated, both from practical and theoretical point of view. In particular, the high dimensional case, namely when $p$ is allowed to be (much) larger than $n$, has attracted a lot of attention. To manage the estimation of $\beta$, or the corresponding prediction $X\beta$, several approaches have been proposed. We can mention, among others, the Lasso algorithm introduced in \textcite{Tibshirani} or the elastic net method discussed in \textcite{Elastic2005}, both being based on a penalization of the least square (LS) problem. Another way to regularize the problem is to introduce a dimension reduction step. For instance, a Principal Component Analysis (PCA) performed on the design matrix $X$ will allow to reduce the number of explanatory variables: such a principle give rise to the Principal Component Regression (PCR). For a comprehensive introduction to this domain and to the aforementioned approaches, we refer to \textcite{Giraud} or \textcite{ESLI}.

This paper deals with the Partial Least Square (PLS) Algorithm (see for instance \textcite{Hoskuldsson}). The main difference with PCR relies in the fact that the dimension reduction step is not only driven by the design matrix $X$ but also by the response vector $Y$. Although this method and its variants have been widely used in an application purpose (in genetics \parencite{LeCao}, social science see \parencite{sawatsky2015partial}, in medicine  \parencite{yang2017application} as a short sample for possible references) and in particular in chemometrics (\textcite{SWOLD2001} or \textcite{AlSouki} among others), the non-linearity of the PLS algorithm makes its statistical analysis difficult. The aim of this paper is to provide a sharp description of the theoretical performances of this method in terms of associated prediction. In particular, denoting by $\hat\beta_{PLS}$ the PLS estimator of $\beta$, we provide a non-asymptotic bound for the prediction risk 
\begin{equation}\label{eq:predict}
\frac{1}{n} \| X\hat \beta_{PLS} - X\beta\|^2,\end{equation}
under a minimal set of assumptions. This bound extends a previous asymptotic analysis conducted in \textcite{Cook2019}. Considering a non-asymptotic framework allows to exhibit several scenarios for which the estimator provides or not relevant performances. In particular, we highlight some specific regimes where the signal to noise ratio is not large enough to ensure a control of the prediction risk. We prove that we can solve this problem by introducing a Ridge regularization step. These investigations are illustrated by numerical simulations. 

The paper is organized as follows. We provide a description of the PLS algorithm in Section \ref{s:PLS}. A prediction bound for the classical PLS estimator is provided and discussed in in Section \ref{s:main1}. Its regularized counterpart is investigated in Section \ref{s:main2}. Proof and technical results are gathered in Sections \ref{s:technical}, \ref{s:proof} and  \ref{s:proof_ridge}.

All along this contribution, we will use the following notations and conventions. The design matrix $X$ is considered as deterministic. The associated Gram matrix is written $\Sigma=\frac{1}{n} X^TX$.  We denote by $\hat \sigma=\frac{1}{n}X^TY$ the empirical covariance between $X$ and the response vector $Y$. The so-called population version of this last quantity is written $\sigma=\mathbb{E}(\hat\sigma)$ where $\mathbb{E}$ denotes the expectation w.r.t.$\ \varepsilon$. Given a matrix $A \in \mathbb{R}^{p\times s}$, $[A]:=\mathrm{span}(A)$ denotes the subspace of $\mathbb{R}^p$ generated by the columns of $A$. If $A \in \mathbb{R}^{s\times s}$ is a positive definite matrix, the highest and the lowest eigenvalues will be denoted respectively by $\rho(A)$ and $\rho_{\min}(A)$, its trace by $\Tr(A)$, while its condition number writes $\Cond(A)$. The diagonal matrix $\diag(A_{11},\dots, A_{ss})$ extracted from $A$ will be written $\diag(A)$. The $\ell^2$ norm is written $\|.\|$.

\section{The PLS estimator}
\label{s:PLS}
\subsection{The PLS algorithm}

Considering the linear model \eqref{Eq:Modele lineaire} in a high dimensional context, namely when $p$ is allowed to be much larger than $n$, creates mathematical issues since the classical least square estimator
\begin{equation}
\hat \beta_{OLS} = \argmin_{\beta\in \mathbb{R}^p} \| Y - X\beta\|^2,
\label{eq:OLS}
\end{equation}
is no more defined. The spectrum of the matrix $\Sigma$ indeed contains the eigenvalue $0$ with a strictly positive multiplicity. 

To solve this problem, several alternative methods have been proposed over the years. A possible way is to penalize the objective function in \eqref{eq:OLS}. For instance, introducing a $\ell^2$ (resp. $\ell^1$) penalty leads to the Ridge (resp. Lasso) estimator, while mixing both of them gives rise to the elastic-net estimator (see \textcite{ESLI}). As an alternative, one can consider dimension reduction methods, searching for the solution of the least square problem in a given subspace $H \subset \mathbb{R}^p$. More formally, we consider the estimator $\hat\beta_H$ defined as
$$ \hat \beta_H = \argmin_{\beta \in H} \| Y- X\beta\|^2.$$ 
For instance, the subspace $H$ can be based on the PCA decomposition of $X$ and defined as the subspace spanned by the first $K$ eigenvectors of $\Sigma$. Such a construction leads to the so-called Principal Component Regression (PCR). The choice of $K$ often reduces to a bias / variance trade-off. 

For the PCR, the subspace $H$ is only constructed from the design matrix $X$. The Partial Least Square (PLS) approach provides an alternative construction using both $X$ and the response vector $Y$. The PLS method is an iterative algorithm. For the first $K$ iterations with $K \in \lbrace 1,\dots, p\rbrace$, the idea is to look for the components which are the most correlated with the vector $Y$. In particular, for each $k\in \lbrace 1,\dots, K\rbrace$, we solve 
$$ \textbf{w}_k = \argmin_{w \in \mathbb{R}^p} \left[ - \frac{1}{n} \langle Y,\textbf{X}^{(k)}w \rangle \right],$$
where $\textbf{X}^{(k)}$ is a deflated version of $X$ defined iteratively as $\textbf{X}^{(k+1)}\textbf{=}\textbf{X}^{(k)}-\textbf{P}_{[\textbf{t}_k]}(\textbf{X}^{(k)})$ where $\textbf{t}_k=\textbf{X}^{(k)} \textbf{w}_k$ and $\textbf{P}_{[\textbf{t}_k]}$ denotes the orthogonal projection operator over $[\textbf{t}_k]$. The PLS construction is formalized in \cref{Algo:PLS} below.   

\begin{algorithm}[!ht]
\caption{Construction of the PLS components}\label{Algo:PLS}
Input {X}, Y and K
\begin{algorithmic}
\State $\textbf{X}_1\textbf{=}{X}$
\For{k\textbf{=}1,\dots, K}
        \State $\textbf{w}_k\textbf{=}\textbf{X}^{(k)^T} Y/ \| \textbf{X}^{(k)^T} Y\|_2$\quad (loadings computation)
        \State $\textbf{t}_k\textbf{=}\textbf{X}^{(k)}\textbf{w}_k$ \quad (component construction)
        \State $\textbf{X}^{(k+1)}\textbf{=}\textbf{X}^{(k)}\ \textbf{-}\ \textbf{P}_{[\textbf{t}_k]}(\textbf{X}^{(k)})$ \quad(deflation step)
\EndFor
\end{algorithmic}
\end{algorithm}

The vectors $(\textbf{w}_k)_{k=1..p}$ correspond to the PLS loadings while the PLS components are the vectors $(\textbf{t}_k)_{k=1..p}$. Given a number of components $K \in \lbrace 1,\dots, p \rbrace$, the associated PLS estimator $\hat\beta_{K}$ is then defined as
$$\hat\beta_{K} \in \mathrm{arg} \min_{\beta\in [W]} \| Y - X\beta \|^2 \quad \mathrm{with} \quad [W] = \mathrm{span}(\textbf{w}_1,\dots,\textbf{w}_K). $$
The PLS has been widely used in the last decades and can be considered as a cornerstone in applied statistics. It is obviously not possible to provide an exhaustive list of references on the subject. We mention, among others, \textcite{Aparicio}, \textcite{FrankFriedman1993}, \textcite{LeCao}, \textcite{Durif}, \textcite{AlSouki}, \textcite{yang2017application}, \textcite{abdel2014comparison} or \textcite{lee2011sparse}. However, we stress that this contribution has not an application purpose. We propose in the following sections to investigate the theoretical performances of this algorithm in terms of the prediction error \eqref{eq:predict}.

\subsection{Krylov representation and regularization}
The iterative form of \cref{Algo:PLS} makes the statistical analysis of the PLS method difficult. Nevertheless, \textcite{Helland} demonstrated that $[W]=\hat{ \G}$, where $\hat\G$ denotes the Krylov space defined as
$$\hat\G:= [\hat G] \quad \mathrm{with} \quad \hat G =( \hat\sigma, \Sigma\hat\sigma, \dots, \Sigma^{K-1} \hat\sigma).$$
This perspective is better suited to evaluate the theoretical performances of $\hat\beta_{PLS}$. It provides an explicit formula of the $K$ directions of the subspace spanned by the weights without a reference to their iterative aspect. In particular, the prediction associated to the PLS algorithm writes $X\hat\beta_K = P_{[XW]}Y = P_{[X\hat G]}Y$. Moreover, the PLS estimator satisfies
\begin{equation}\label{Eq:PLSestimator}
    \hat\beta_{K}= \argmin_{\beta\in \hat\G} \| Y - X\beta\|^2 = \hat G \hat\Theta^{-1} \hat G^T\hat\sigma \quad \mathrm{where} \quad \hat\Theta = \hat G^T \Sigma \hat G = (\hat\sigma^T \Sigma^{i+j-1} \hat\sigma )_{i,j=1..K},
\end{equation}
provided $\hat\Theta \in\mathbb{R}^{K\times K}$ is full rank. In expression \eqref{Eq:PLSestimator} above, each term is explicit and can be computed from the data.

Following \textcite{Cook2019}, the formulation displayed in \eqref{Eq:PLSestimator} will be a starting point for our analysis. First, we introduce the so-called population version of respectively $\hat G$ and $\hat \Theta$ defined respectively as
$$ G=\Bigl(\sigma, \Sigma\sigma, \dots \Sigma^{K-1}\sigma\Bigr)\in\RR^{p\times K} \quad \mathrm{and} \quad \Theta= G^T\Sigma G= (\sigma^T \Sigma^{i+j-1} \sigma )_{i,j=1..K}\in\RR^{K\times K}.$$
We will in particular focus our attention on the term $\bar{\beta}$ defined as 
\begin{equation}
    \bar{\beta}=G\Theta^{-1}G^T\sigma = G \Lambda \quad \mathrm{with} \quad \Lambda = \Theta^{-1} G^T \sigma,
    \label{Eq:barbeta}
\end{equation}
provided $\Theta$ has full rank. We can remark that $X\bar\beta$ allows to determine the best approximation of $X\beta$ over the image of $[XG]$ on the Krylov space by the design matrix $X$, namely $X\bar\beta = P_{[XG]}X\beta$.

The matrix $\Theta$ will play an important role in our analysis displayed below. The non-singularity of the matrix $\Theta$ implies that $G$ is full rank. The determinant of $\Theta$ represents the volume of the parallelotope formed by the Krylov components. In particular, the components are linearly independent if and only if the parallelotope has non-zero n-dimensional volume. In the following, we introduce 
\begin{equation} 
D=\diag(\Theta) \quad \mathrm{and} \quad R=D^{-\frac{1}{2}}\Theta D^{-\frac{1}{2}}.
\label{eq:Rmatrix}
\end{equation}
The matrix $R$ is the normalized version of $\Theta$ which can be interpreted as a correlation matrix between the Krylov components. If the components are linearly independent, the inversion of $\Theta$ is equivalent to the inversion of $R$. Since the estimator $\hat \beta_{PLS}$ in \eqref{Eq:PLSestimator} involves an estimated version of $\Theta$, it appears that the performances will deteriorate when the smallest eigenvalue $\rho_{\min}(R)$ of  $R$ will be too small in a sense which is made precise in Section \ref{s:main1}. Moreover the estimation of $\Theta$ by the random matrix $\hat\Theta$ and its inversion can create instability in the prediction process when the signal to noise ratio is too low (see \ref{ass:signal} in the next section).

To get round of this problem, we will introduce a Ridge regularization step in the PLS estimator. In particular, we will consider in Section \ref{s:main2} the estimator $\hat\beta_{K,\mathbf{\alpha}}$ defined as
\begin{equation}
\hat\beta_{K,\mathbf{\alpha}} = \hat G (\hat\Theta+\Delta_{\alpha})^{-1} \hat G^T\hat\sigma \quad \mbox{for any} \quad \Delta_{\alpha}=\mathrm{diag}(\alpha_1,...,\alpha_K) \in \mathbb{R}^K.
\label{eq:betaridge}
\end{equation}
We prove that with an appropriate choice of $\mathbf{\alpha}$, the bounds are similar compared to the case where the signal to noise ratio is large enough.

The non-linearity of the PLS estimator \eqref{Eq:PLSestimator} may explain that few investigations have been conducted regarding its theoretical performances. In the single component case, we can mention the seminal contribution proposed in \textcite{Cook}, where - up to our knowledge - bounds in terms of prediction error where proposed for the first time in an asymptotic context. Extensions of theses results have been proposed in e.g., \textcite{BCFM_22} with less restrictive conditions on the parameter $\beta$, asymptotic normality for $\hat\beta_1$ and confidence intervals, or in \textcite{Cast} where a non-asymptotic study was conducted and a sparse version of $\hat\beta_1$ has been considered.  In the general case (namely when $K\in \lbrace 1,\dots, p\rbrace$), we refer to \textcite{Cook2019} for asymptotic investigations. This paper can be considered as a generalization to the non-asymptotic case. Considering a non-asymptotic setting allows in particular to exhibit some additional specific scenarios for which prediction may not be pertinent. For instance, difficulties regarding the inversion of $\Theta$ may be hidden in an asymptotic context. This may help to describe the advantages and limitations of this method.  Non asymptotic bounds in prediction are proposed for both $\hat\beta_{K}$ (Section \ref{s:main1}) and its regularized version $\hat\beta_{K,\mathbf{\alpha}}$ (Section \ref{s:main2}).

\section{Theoretical results}
\label{s:main1}
Our first contribution is an upper bound on the quadratic loss in prediction for the PLS estimator with $K$ components. This bound is derived explicitly, subject to certain assumptions regarding the Krylov components. With the Krylov space point of view, we take into account the energy of each component in order to guarantee a non-asymptotic control.

\subsection{Non-asymptotic analysis}

The PLS estimator requires the inversion of the Gram matrix $\hat{\Theta}$ (defined in \eqref{Eq:PLSestimator}). The invertibility of this matrix ensures that $\hat{G}$ is full rank, that is, the Krylov components are linearly independent. To this end, we will make assumptions on the space spanned by $G$ via assumptions on the matrix$~\Theta$. As the inversion of $\Theta$ is equivalent to the inversion of $R$, we express this hypothesis on the correlation matrix $R$ introduced in \eqref{eq:Rmatrix}.

\begin{description}[font=\bf]
\item[Assumption A.1.]\namedlabel{ass:inverse}{{Assumption A.1}} \textit{The correlation matrix $R$  verifies $\rho_{\min}(R)>0$}.
\end{description}

Additionally, ensuring an accurate estimate of $\Theta^{-1}$ is essential. 
Recall that the diagonal elements of $\Theta$, which represent the norm of the Krylov components, are given by 
$$\sigma^T\Sigma^{2i-1}\sigma=\frac{1}{n}\|X\Sigma^{i-1}\sigma\|^2, \quad \forall i=1,\dots,K.$$
We consider the PLS regression with $K$ components if all the $K$ norms above are large enough. That is, higher than a value which will be defined later on. Intuitively if a norm is almost zero, then so is the Gram determinant of $\Theta$ despite the fact the matrix $G$ is full rank. The justification for this assertion is based on the classic Hadamard inequality which bounds the volume of the parallelotope by the products of the norms of the Krylov components. Making the assumption that these quantities are above a certain level of fixed inertia is then natural for achieving a good estimation of the matrix and obtaining efficient bounds for $K$ components. These levels of inertia are related to the variance terms of the estimators $\hat{\sigma}^T\Sigma^{2i-1}\hat{\sigma}$ of the Krylov components norms in the next assumption (see also \cref{Cor:Majoration3Termes} in Appendix \ref{s:technical}).

\begin{description}[font=\bf]
\item[Assumption A.2.]\namedlabel{ass:signal}{{Assumption A.2}} \textit{Let $\delta\in[0,1]$. For all $i\in\{1,...,K\}$, }
\begin{equation}
\sigma^T\Sigma^{2i-1}\sigma  \geq t_{\delta,R}\,\frac{\tau^2}{n} K\rho(\Sigma)^{i}\Tr(\Sigma^{i}) 
\quad \mathrm{with} \quad t_{\delta,R} \geq  128\frac{\ln(6K/{\delta})}{\rho_{\min}(R)}.
\label{eq:ass21}
\end{equation}
\end{description}

As observed by \textcite{Cook2019}, PLS estimation can be done even with a non invertible matrix $\Sigma$. Our conditions deal with the components of the Krylov space. First, \ref{ass:inverse} say that the components are linearly independent. It means the dimension $K$ is chosen sufficiently small, so that there is no redundant information. This assumption is not restrictive. \ref{ass:signal} is more restrictive. It ensures that the signal-to-noise ratio corresponding of each component is high enough. To highlight the interpretation of our condition \ref{ass:signal}, suppose similarly to \textcite{Cook2019} that the Gram matrix $\Sigma$ decomposes as
\begin{equation}\label{eq:decompS}\Sigma= H\Sigma_H H^T+H_0\Sigma_{H_0} H_0^T,\end{equation}
with $\bar{\beta}=G(G^T\Sigma G)^{-1}G^T\sigma = H(H^T\Sigma_H H)^{-1}H^T\sigma$. 
One can easily show that a sufficient condition to have \ref{ass:signal} is that $\max\Big(1,\big(t_{\delta,R}K\frac{p}{n}\frac{\tau^2}{\beta^T\Sigma\beta}\big)^{1/2}\Big)\frac{\rho(\Sigma)}{\rho_{\min}(\Sigma_H)}\leq 1$. That is, the minimal inertia of a component in the Krylov space should not be negligible beyond the maximal inertia of $\Sigma$. Observe that \textcite{Cook2019} assume that $\beta^T\Sigma\beta$ is bounded and $\frac{\Tr(\Sigma)}{n\Tr(\Sigma_H)}$ goes to 0. The two conditions are not directly related but they both read as a sufficient part of inertia of $\Sigma_H$.

\ref{ass:signal} guarantees that the matrix $\hat\Theta$ is invertible, as displayed in \cref{lem:R} (see Appendix \ref{s:technical}). It can be considered as a signal to noise ratio condition that ensures that enough signal is available in the considered theoretical  components to use the PLS algorithm on the couple $(X,Y)$. This assumption can not be checked in practice from real data since it is strongly related to the covariance $\sigma$. This assumption makes sense given the general approach adopted for this problem without any additional assumption on the matrix $\Sigma$. For the single component case in \citet{Cast}, this assumption is in line with the discussion made about the norm of the first component with the "high signal case" corresponding to 
$$\sigma^T\Sigma\sigma\ge t_{\delta}\frac{\tau^2}{n}\rho(\Sigma)\Tr(\Sigma).$$

In Section \ref{s:main2} we will introduce an alternative procedure with a Ridge regularization on matrix $\Theta$. We will prove that this approach allows to remove \ref{ass:signal}, up to an additional bias term in the prediction error.

Now, we have all the ingredients to present our first main result which provides a non-asymptotic bound on the prediction error of the PLS estimator. 

\begin{theorem}\label[thm]{Th:1} 
Let $\delta \in (0,1)$. Suppose that \ref{ass:inverse} and \ref{ass:signal} hold. Then, with a probability larger than $1-\delta$, 
\begin{multline*}
    \frac{1}{n}\|X\hat{\beta}_{K}-X\beta\|^2
    \le\frac{2}{n}\underset{v\in[G]}{\inf}\|X(\beta-v)\|^2\\
+D_{\delta,R}\frac{\tau^2}{n}\max\bigg(\Cond(D)\|\Lambda\|^2\sum_{i=1}^{K}\Tr(\Sigma^{2i}),\sqrt{\Cond(D)\|\Lambda\|^2\sum_{i=1}^{K}\Tr(\Sigma^{i})^2}\bigg),
\end{multline*}
for some constant $D_{\delta,R}$ depending only from $\delta$ and $R$.
\end{theorem}

An explicit expression of the constant $D_{\delta,R}$ is given in the proof presented in Section \ref{s:proof} (see equation \eqref{eq:constantD}). We stress that the result displayed in \cref{Th:1} has been simplified for the ease of exposition. A more precise result has actually been proven (see in particular Section \ref{s:Bonus}).

The bound on the prediction error displayed in \cref{Th:1} relies on deviation result on non-centered weighted $\chi^2$ distribution with matrix norm inequalities. The analysis of this first result is discussed in the next subsection.

\subsection{Discussion} \label{sec:discuss}

The bound displayed in \cref{Th:1} is composed of two different terms which describe the classical bias-variance trade-off. The first term 
$$ \frac{1}{n}\underset{v\in[G]}{\inf}\|X(\beta-v)\|^2,$$
corresponds to the bias. It measures the distance between the true signal $X\beta$ and the most accurate prediction in the Krylov subspace $[G]$. This term depends on $K$ through the dimension of $\G$. In particular, using a large number of PLS components (i.e. a large value of $K$) will allow to provide a good approximation (understood as an approximation associated to a small bias). On the other hand, using few PLS components may lead to the situation where the Krylov space $[G]$ cannot provide a good approximation of the target $X\beta$. The second term appearing in the r.h.s. of the bound measures the variability of the estimator and can be considered as some kind of variance term. It essentially depends on four main quantities: the smallest eigenvalue of $R$ which measures the correlation between the Krylov components, $\Cond(D)$ which measures the difference of norms between the Krylov components, the trace of a power of the Gram matrix $\Sigma$ and the norm of the Krylov components $ \Lambda$ introduced in \eqref{Eq:barbeta}. The sum of traces of powers of $\Sigma$ and the quantity $\Cond(D)$ increase with $K$. An optimal value for this parameter should provide an equilibrium inside this bound. Note that constructing a data-driven choice for $K$ is beyond the scope of this paper.  

This result perfectly matches with a previous bound obtained in \textcite{Cast} in the specific case where $K=1$. In such a setting, the so-called variance term can be related to
$$ \|\Lambda\|^2 \Tr(\Sigma^2) = \frac{\Tr(\Sigma^2)}{\lambda^2} \quad \mathrm{with} \quad \lambda = \frac{\sigma^T\Sigma\sigma}{\sigma^T\sigma}.$$
This last quantity can be seen as the inverse of a relative inertia in this specific case. We refer to the aforementioned reference for an extended discussion on the subject. 

Following \textcite{Cook2019}, using decomposition \eqref{eq:decompS}, we can express the bound of \cref{Th:1} with the spectra of $R$, of $\Sigma$ and of $\Sigma_H$. Indeed, $\|\Lambda\|^2\leq \frac{\Cond(D)}{\rho_{\min}(R)^2}\sum_{i=1}^K \frac{1}{\rho_{\min}(\Sigma_H)^{2i}}$ and $\Cond(D)$ can be expressed with the spectrum of $\Sigma_H$, provided \eqref{eq:decompS} holds.  The resulting bound  differs from the rate given in \textcite{Cook2019}. This is mainly due to the facts that we consider a fixed design and a non asymptotic framework. 

\cref{Th:1} provides different bounds compared to those displayed in \textcite{Cook2019} where the performances of the PLS estimator are partially described in terms of the trace of $R$ and of the shape of the spectrum of $\Sigma$. Although we start with the same risk decomposition, we provide a non-asymptotic investigation: we do note require $n$ (or $p$) going toward infinity. Moreover, we do not use any assumption on the structure of $\G$ (except that it has a dimension $K$). We do not suppose for instance that it exactly handles $\beta$, contrarily to \textcite{Cook2019} where $\beta=\bar{\beta}$ defined in \eqref{Eq:barbeta}.

Last but not least, we consider fixed covariates $X$ (and hence fixed Gram matrix $\Sigma$) while \textcite{Cook2019}'s setting deals with random covariates. It enables to highlight the influence of the Krylov components. As illustrated by \ref{ass:signal}, this may create some issues when these quantities are close to the standard deviation of their estimates. To overcome this problem, a regularization step may be used. This will be considered in the next section.

\section{Estimation with Ridge PLS estimator }
\label{s:main2}
In order to avoid \ref{ass:signal} that ensures a control on the error driven by the estimation of $\Theta$ and its inversion,  we have introduced a variant of the PLS estimator that involves a Ridge penalization as displayed in \eqref{eq:betaridge}.

The Ridge estimator in model \eqref{Eq:Modele lineaire} was first introduced by \textcite{hoerl1970ridge}.  \textcite{theobald1974,farebrother1976} or more recently \textcite{dobriban2018} showed its efficiency in prediction. Here, we consider a Ridge approach with different penalties for the components. We refer to \textcite{van2021fast} and references therein for such an approach in regression models. More generally, a review of regularization approaches for covariance matrices, not specifically for regression models, including Ridge approach with multiple penalties, can be found in \textcite{engel2017overview}.

In some sense, we force the invertibility of the $\hat{\Theta}$ by summing it with a diagonal matrix $\Delta_\alpha=\mathrm{Diag}(\alpha_1,...,\alpha_K)$. Such a regularization avoids in particular to call upon \ref{ass:signal} to obtain a control on $\rho_{\min}(\hat\Theta)$ and $\rho(\hat\Theta)$ (see \cref{lem:invertR} for more details). \cref{Th:2} below provides a specific choice for the regularization parameters $\alpha$ and describes the performances of the estimator $\hat\beta_{K,\alpha}$ introduced in \eqref{eq:betaridge}. The associated proof is postponed to Section \ref{s:proof_ridge}.

\begin{theorem}\label[thm]{Th:2} 
Let $\delta \in (0,1)$. Suppose that \ref{ass:inverse} holds and set
\begin{equation}
\alpha_i = c_{\delta}K\frac{\tau^2}{n}\rho(\Sigma)^{i}\Tr(\Sigma^{i}) \quad \forall i\in \lbrace 1,\dots, K\rbrace \quad \mathrm{with} \quad c_{\delta} = 16C_\delta,
\label{eq:alphai}
\end{equation}
where $C_\delta$ is made precise in \cref{Cor:Majoration3Termes}. Then, with a probability larger than $1-\delta$, 
\begin{multline*}
    {\frac{1}{n}\|X\hat{\beta}_{K,\alpha}-X\beta\|^2}
     \le\frac{2}{n}\underset{v\in[G]}{\inf}\|X(\beta-v)\|^2\\
    +D_{\delta,R}'\frac{\tau^2}{n} \max\bigg(\Cond(D)\|\Lambda\|^2\,K\sum_{i=1}^{K}\rho(\Sigma^{i})\Tr(\Sigma^i),\sqrt{\Cond(D)\|\Lambda\|^2\sum_{i=1}^{K}\Tr(\Sigma^{i})^2}\bigg).
\end{multline*}
where 
\begin{equation}\label{eq:g2}
    D_{\delta,R}' = c_{\delta}' \Cond(R)^{4},
\end{equation}
and $c_{\delta}'$ is a positive constant depending on $\delta$.
\end{theorem}

\cref{Th:2} provides a result similar to the bound displayed in \cref{Th:1}. The choice of the $\alpha_i$ is related to the variance of the diagonal terms of the matrix $\hat\Theta$. In some sense, the regularization allows to counterbalance the effect of the noise that may deteriorate the rank. 

The introduction of this regularization term allows to remove \ref{ass:signal} from our analysis. We recall that this assumption is related to a sort of signal-to-noise ratio that should be large enough to guarantee good performances for the PLS estimator (see the previous section for an extended discussion). Since this ratio is not known a priori, our approach allows to secure the prediction. Nevertheless, we stress that the theoretical calibration of the $\alpha_i$ involves unknown constants such as $\tau^2$. From a practical point of view, a typical approach to get round of this problem would be to use a data-driven calibration (as, e.g., a cross-validation procedure).

To conclude this discussion, we point out that our regularization is strongly related to the Krylov representation of the PLS estimator. In particular, it is related to the following optimization problem.
\begin{proposition}
	 We have,
 	$$\hat{\beta}_{K,\alpha}=\hat{G}\cdot\underset{u\in\RR^{K}}{\mathrm{argmin}}\|Y-X\hat{G}u\|^2+u^T\Delta_{\alpha}u.$$
\end{proposition}
\begin{proof}
	The function $g(u):=\frac{1}{n}\|Y-X\hat{G}u\|^2+u^T\Delta_{\alpha}u$ is convex and differentiable. The minimum satisfies the equality $\grad g(u)=\hat{G}^T\Sigma\hat{G}u+\Delta_{\alpha}u=\hat{G}^T\hat{\sigma}.$ It yields $\hat{\Theta}_{\alpha}u=\hat{G}^T\hat{\sigma}$. We deduce $$\hat{\beta}_{K,\alpha}=\hat{G}\hat{\Theta}_{\alpha}^{-1}\hat{G}^{T}\hat{\sigma}.$$
\end{proof}

Note that the computation of the estimator can be done either by the optimization problem, or using the explicit formulation. It only involves $K\times K$ matrices, since the reduction of dimension has been done through $\hat G$.

The term $u^T\Delta_{\alpha}u$ in this optimization problem is a $\ell^{2}$ penalization weighted by the $\alpha_i$. It operates on the Krylov coordinates of the estimator and not on the estimator itself. In particular, replacing $u^T\Delta_{\alpha}u$ by $ u^T\hat G^T \Delta_{\alpha} \hat G u $  will leads to the classical Ridge estimator (still restricted on $\hat G$) but will not allow to control the minimal eigenvalue of $\hat R$. 

The following section provides numerical simulation in a toy setting. In particular, it allows to prove that the unstability of the PLS is not only a mathematical artefact related to the Krylov representation and that the regularization proposed in this paper allows to improve the performances of the PLS estimator.

\section{Simulation study}\label{s:main3}

In this section we illustrate the properties of the Ridge PLS estimator. In a first time we define $\beta=\bar{\beta}$ as a linear combination of two normalized eigenvectors of the covariance matrix $\Sigma$. This guarantees that the bias term of \cref{Th:2} ($\frac{1}{n}\|X(\beta-\bar{\beta})\|^2$) is equal to zero allowing us to focus on the variance term. In particular, we study the effect of the signal-to-noise ratio (corresponding to \ref{ass:signal}) on the standard PLS estimator, and the effect of the Ridge regularization.
In a second time, we illustrate the bias variance tradeoff thanks to a parameter representing the distance between $\beta$ and the theoretical Krylov subspace.

We generate $N=2000$ samples of size $n=200$ as follows. We consider the case with $p=5$ with an underlying space $\G$ of dimension 2. It does not correspond to a high-dimensional setting but this framework allows to highlight more easily the behavior of the estimators with respect to the eigen structure of the Gram matrix $\Sigma$. For each simulation, we generate a design matrix $X\in \mathbb{R}^{n\times p}$ from a Gaussian centered distribution such that $\Sigma=\diag(\lambda_1,\lambda_2,\lambda_3,\lambda_4,\lambda_5).$ 
We denote $v_{i}$ the eigenvector of $\Sigma$ associated with the eigenvalue $\lambda_i$, for $i=1,\dots,5$. 
For a given $\beta$, the response $Y$ is generated according to \eqref{Eq:Modele lineaire} with $\tau^2=1$. Covariates $X$ are fixed among the $N$ samples while the noise $\varepsilon$ varies. Different scenarios will be considered, using different definitions of $\beta$.

\subsection{Effect of the regularization}\label{subsec:Regularization}

First, we consider a case without bias. We introduce a parameter $\eta\in \mathbb{R}^{+}$ in the definition of $\beta$ which corresponds to a signal-to-noise ratio. 

\paragraph{Scenario 1.}
We compute $\beta=\eta\cdot\big(v_{1}+v_{2}\big)$, with  $\eta>0$. We consider the following values of eigenvalues:
\begin{itemize}[label={},topsep=0pt]
    \item \textbf{Scenario 1a.} $\lambda_1=6.1$, $\lambda_2=6$, $\lambda_3=\lambda_4=\lambda_5=0.5$,
    \item \textbf{Scenario 1b.} $\lambda_1=0.9$, $\lambda_2=0.3$, $\lambda_3=\lambda_4=\lambda_5=0.2$.
\end{itemize}
\paragraph{Scenario 2.}
We compute $\beta=\eta\cdot\big(v_{4}+v_{5}\big)$, with  $\eta >0$. We consider the following values of eigenvalues:
\begin{itemize}[label={},topsep=0pt]
\item \textbf{Scenario 2a.} $\lambda_1=3$, $\lambda_2=2$,  $\lambda_3=2$, $\lambda_4=2$, $\lambda_5=1$.
    \item \textbf{Scenario 2b.} $\lambda_1=4$, $\lambda_2=2$, $\lambda_3=2$, $\lambda_4=2$, $\lambda_5=1$,
\end{itemize}

The configuration of Scenario 1 is such that the Krylov subspace is carried by the two main eigenvalues of the matrix $\Sigma$. Scenario 2 corresponds to a case where the Krylov components are carried by small eigenvalues of $\Sigma$.

In Scenario 1, when $\beta=\eta\cdot\big(v_{1}+v_{2}\big)$, the theoretical covariance $\sigma$ satisfies
\begin{align*}
    \sigma=&\ \eta \cdot \big(\lambda_{1}v_{1}+\lambda_{2}v_{2}\big),\\
   % \Sigma\sigma=&\ \eta \cdot \big(\lambda_{1}^{2}v_{1}+\lambda_{2}^{2}v_{2}\big),
     \sigma^T\Sigma^{2i-1}\sigma=&\ \eta^2 \cdot \big(\lambda_{1}^{2i+1}+\lambda_{2}^{2i+1}\big), \quad i=1,2.
\end{align*}
Calculating $\Theta$ gives 
$$\Theta=\eta^2\cdot\begin{pmatrix}
\lambda_1^3+\lambda^3_2 & \lambda_1^4+\lambda_2^4\\
\lambda_1^4+\lambda_2^4 & \lambda_1^{5}+\lambda_2^{5}
\end{pmatrix}.
$$
We compute $\Lambda$ as a function of $\lambda_1$ and $\lambda_2$, %which gives us 
\begin{equation}
\Lambda=\left(\frac{\lambda_1+\lambda_2}{\lambda_1\lambda_2},-\frac{1}{\lambda_1\lambda_2}\right).\label{Eq:LambdaK2}
\end{equation}
Note that the Krylov coordinates $\Lambda$ are independents of $\eta$. The parameter $\eta$ preserves $\Lambda$ while modifying the determinant of $\Theta$. It is clear that for a given $\delta$, \ref{ass:signal} is not satisfied for low values of $\eta$, and is satisfied for high values.

Similar equations can be displayed in Scenario 2, with a change in the indexes.

Recall that $K=2$ in this section. Our aim is to compare the numerical performances of the PLS estimator 
\begin{equation}
\hat{\beta}_{K}=\hat{G}\hat{\Theta}^{-1}\hat{G}^T\hat{\sigma},\label{Eq:SimusPLSEstimator}
\end{equation}
and its regularized version
\begin{equation}
\hat{\beta}_{\alpha}=\hat{G}\hat{\Theta}_{\alpha}^{-1}\hat{G}^T\hat{\sigma}.
\label{Eq:SimusRidgeEstimator}
\end{equation}
To emphasize the effect of the inversion of $\hat\Theta$ on the estimation process, we also include in the analysis the pseudo-estimator $\hat\beta^{or}$ defined as
\begin{equation}
\hat{\beta}^{or}=\hat{G}\Theta^{-1}G^T\sigma.
\label{Eq:SimusAxisEstimator}    
\end{equation}

The pseudo-estimator $\hat{\beta}^{or}$ correspond to the specific case where the Krylov coordinates $\Lambda$ are assumed to be known and can be considered as an oracle. It is linear in the direction of the subspace $\hat{\G}$. The quadratic risk associated to $\hat{\beta}^{or}$ does not depend of the parameter $\eta$. Indeed, its quadratic risk is equal to $\Lambda^T(\hat{G}-G)^T\Sigma(\hat{G}-G)\Lambda$. Equation \eqref{Eq:LambdaK2} shows that $\lambda$ does not depend on $\eta$. The $j\textsuperscript{th}$ column of $\hat{G}-G$ is 
$\Sigma^{i-1}\hat{\sigma}-\Sigma^{i-1}\sigma=\Sigma^{i-1}\frac{X^T\varepsilon}{n},$
which does not depend on $\eta$ either.
This proves that the risk of $\hat{\beta}^{or}$ is constant as a function of $\eta$.

The parameter $\alpha$ in the estimator $\hat{\beta}_{\alpha}$ is of the form $$\big(\alpha_1,\alpha_2\big):=\bigg(C_1K\frac{\tau^2}{n}\rho(\Sigma)\Tr(\Sigma),C_{2}K\frac{\tau^2}{n}\rho(\Sigma)^{2}\Tr(\Sigma^{2})\bigg),$$
with $C_1$ and $C_{2}$ detailed depending of the simulations. 
They were set respectively to $C_1=0.08$ and $C_2=0.05$ in Scenario 1a., and to $C_1=C_2=0.02$ in Scenario 1b. and and respectively to $C_1=0.002$ and $C_2=0.0005$ in Scenario 2a and Scenario 2b. Note that these constants were not modified depending on $\eta$.

Finally, we study the performances of the estimators by the evaluation of $$\mathrm{MSE}_{\eta,j}=\frac{1}{n\times N}\sum_{i=1}^{N}\|X(\beta_{\eta}-\hat{\beta}_{i,j})\|^2,$$
where $\hat{\beta}_{i,j}$ is the estimator according to the \(i\)\textsuperscript{th} sample from $Y=X\beta+\varepsilon$ with $j$ denoting the choice of the estimator. The index $j=1,2,3$ are respectively the PLS estimator \eqref{Eq:SimusPLSEstimator}, the oracle estimator \eqref{Eq:SimusAxisEstimator} and the Ridge estimator \eqref{Eq:SimusRidgeEstimator}.

\begin{figure}[!ht]
    \begin{subfigure}[b]{0.425\textwidth}
    \includegraphics[scale=0.44]{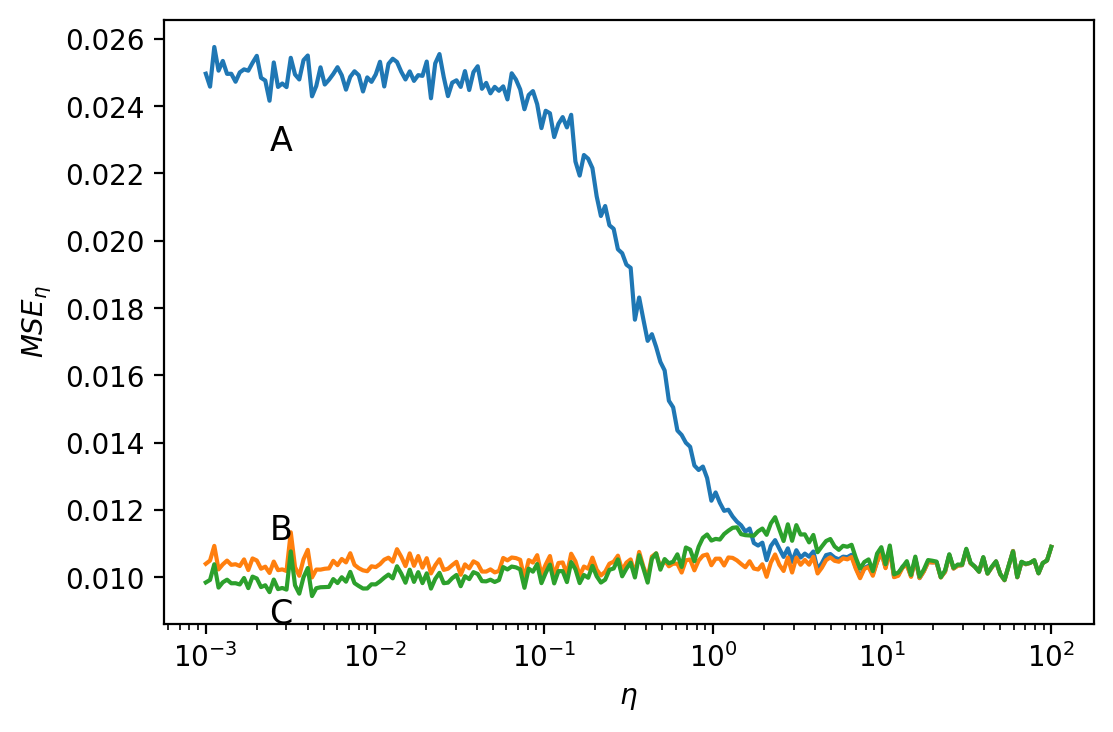}
    \captionsetup{labelformat=simple, labelsep=period}
    \caption{Scenario 1a}
    \label{Risque1}
    \end{subfigure}
    \hspace{0.1\textwidth}
    \begin{subfigure}[b]{0.425\textwidth}
    \includegraphics[scale=0.44]{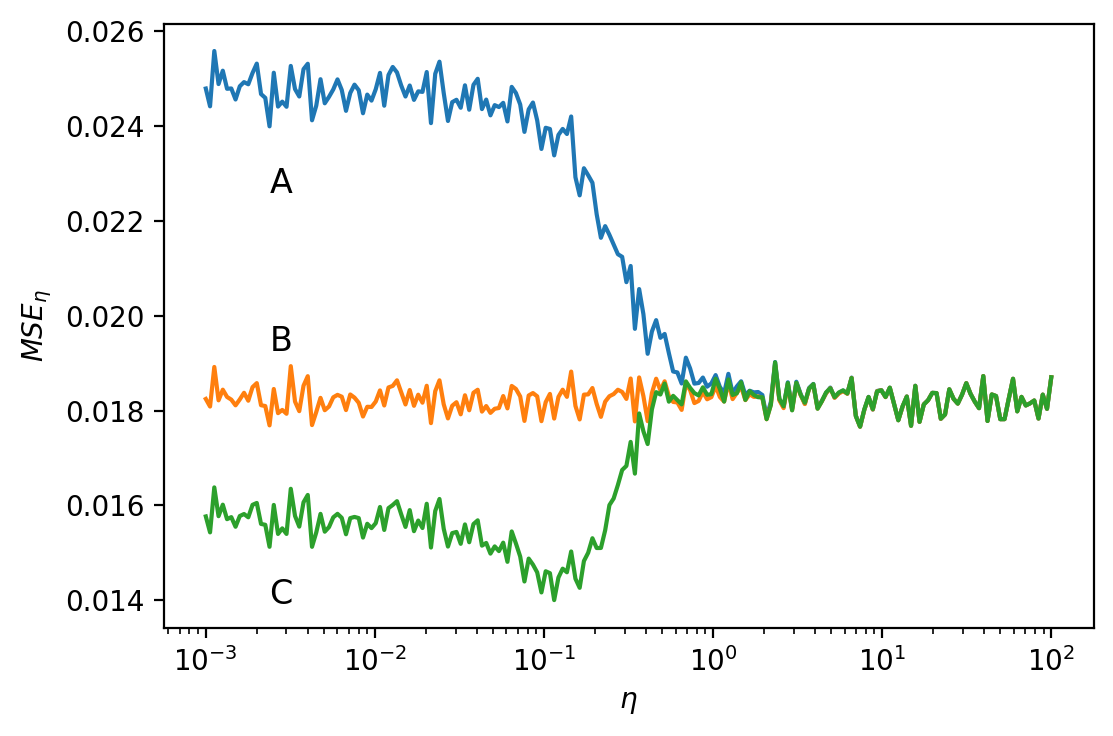}
    \captionsetup{labelformat=simple, labelsep=period}
    \caption{Scenario 1b}
    \label{Risque2}
    \end{subfigure}
    \caption{Quadratic risk $\mathrm{MSE}_{\eta}$ with respect to $\eta$ for Scenario 1. Curves A, B C give the quadratic risk respectively for the PLS estimator \eqref{Eq:SimusPLSEstimator}, the oracle estimator \eqref{Eq:SimusAxisEstimator} and the Ridge estimator \eqref{Eq:SimusRidgeEstimator}.}
    \label{RisqueSimu}
\end{figure}
   
\begin{figure}[!ht]
    \begin{subfigure}[b]{0.425\textwidth}
    \includegraphics[scale=0.44]{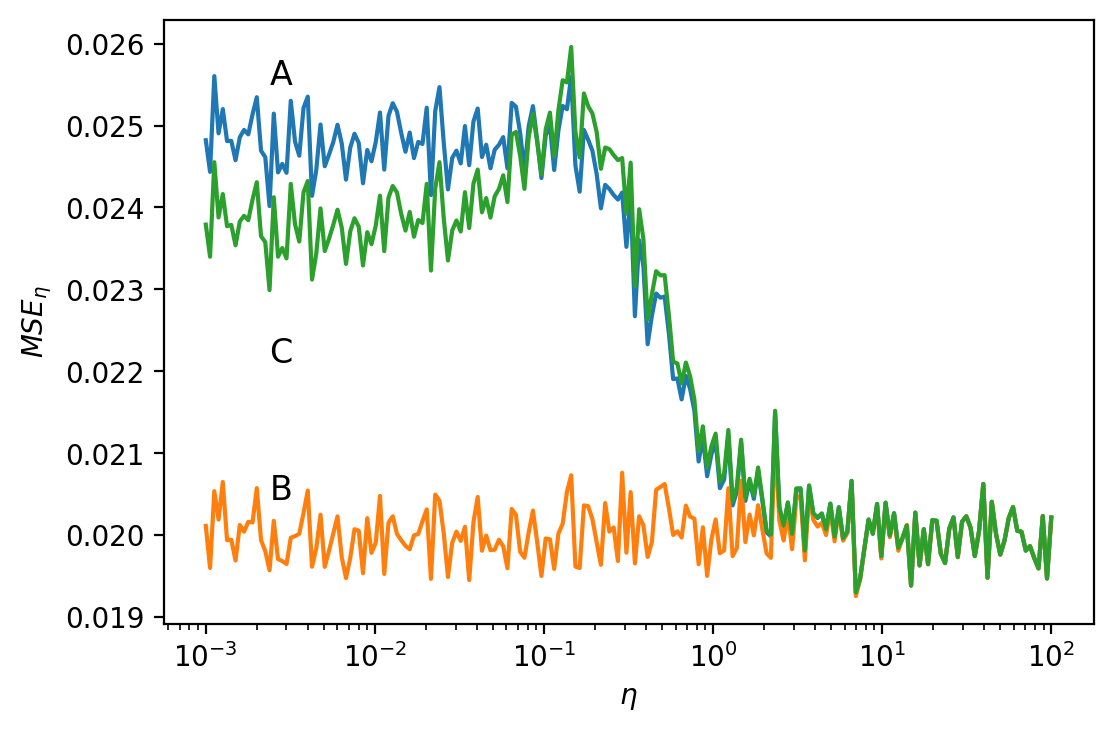}
    \captionsetup{labelformat=simple, labelsep=period}
    \caption{Scenario 2a}
    \label{Risque31}
    \end{subfigure}
    \hspace{0.1\textwidth}
    \begin{subfigure}[b]{0.425\textwidth}
    \includegraphics[scale=0.44]{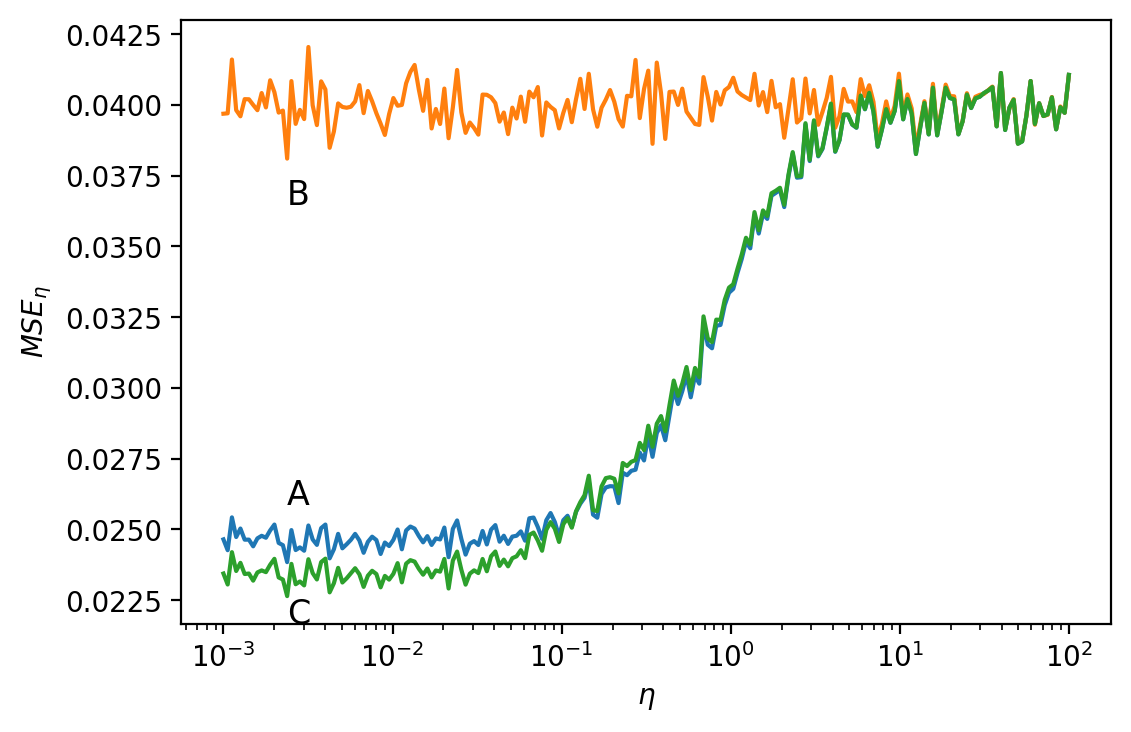}
    \captionsetup{labelformat=simple, labelsep=period}
    \caption{Scenario 2b}
    \label{Risque32}
    \end{subfigure}
    \caption{Quadratic risk $\mathrm{MSE}_{\eta,.}$ with respect to $\eta$ for Scenario 2. Curves A, B C give the quadratic risk respectively for the PLS estimator \eqref{Eq:SimusPLSEstimator}, the oracle estimator \eqref{Eq:SimusAxisEstimator} and the Ridge estimator \eqref{Eq:SimusRidgeEstimator}.}
    \label{RisqueSimu2}
\end{figure}

\subsubsection*{Influence of level-to-noise ratio}

\cref{RisqueSimu} and \cref{RisqueSimu2} display the different quadratic risk associated with each estimator according to $\eta$ on a logarithmic scale, respectively in Scenario 1 and in Scenario 2.

Scenario 1 illustrates that, when the signal-to-noise ratio parameter $\eta$ is low, the quality of the PLS estimator deteriorates. In these settings, the benefits of Ridge regularization is noticeable. In particular when \ref{ass:signal} is not satisfied (for small $\eta$). 

The oracle estimator \eqref{Eq:SimusAxisEstimator} corresponds to an estimator where the PLS axis $\Lambda$ are known and only the coordinates of $\beta$ on the axis are estimated. As the prediction error is constant (up to Monte-Carlo error), it shows that the quality of estimation mainly depends on the quality of the estimated axis. In particular, the degradation of the PLS for low $\eta$ is based on the estimation of $\Lambda$, mainly through the error on $\hat\Theta^{-1}$. The Ridge regularization improves this estimation.

Scenario 1 corresponds to cases where $\beta$ belongs to a Krylov space generated by the two highest eigenvectors of the Gram matrix $\Sigma$.
While Scenario 2 corresponds to cases generated by the two lowest eigenvectors. As illustrated on \cref{RisqueSimu2}, the behaviour of the estimators does not only rely on the rank of the eigenvectors.

Scenario 2a behaves similarly than Scenario 1.
In Scenario 2b, we can observe from \cref{Risque32} that the quadratic risk increases when $\eta$ increases. This last setting is, hence, very different from the others. 
Scenario 2 illustrates two very different behaviors with equal coordinates $\Lambda$. The main difference lies in the spectral radius $\rho(\Sigma)$, i.e. the spectrum of the matrix $\Sigma$. The risk of the pseudo estimator (\ref{Eq:SimusAxisEstimator}) is sensitive to the spectrum as shown in \cref{RisqueSimu2}. Indeed, the only difference between Scenario 2a and Scenario 2b is that the relative inertia explained by the Krylov space differs. To illustrate this inertia ratio between spectrum and Krylov coordinates, we then propose to illustrate in Section \ref{subsec:Tradeoff}, with a fixed spectrum for $\Sigma$, the behaviour of the PLS estimator when the coordinates of $\Lambda$ vary. 

When the inertia ratio is low, in Scenario 2b, the quality of estimation deteriorates (the oracle estimator has a mean quadratic risk of 0.04, compared to 0.02 and lower in other settings). Surprisingly, in this case, the PLS estimator outperforms the oracle estimator, and the quality is equivalent to the Ridge estimator. Such behavior does not occur in settings like Scenario 1, with the Krylov space carried by the main eigenvectors. It only occurs when the ratio of the highest eigenvalue of $\Sigma$ and the lowest eigenvalue in the Krylov space is high ($\rho(\Sigma)/\rho_{\min}(\Sigma_H)$ in Section \ref{sec:discuss}). In this case, the error of projection is high, as shown by the behavior of the oracle estimator.

\subsubsection*{Levels of penalization}

In examples above, the constants $C_1$ and $C_2$ have been appropriately chosen to illustrate the benefits of the Ridge estimator \eqref{Eq:SimusRidgeEstimator}. We propose to highlight the extreme behavior associated with these parameter choices in Scenario 1a. First, the Ridge parameters are set to a low value, that is $C_1=C_2=0.002$, and then they are set higher, $C_1=C_2=0.2$.
The choice of these constants is related to a bias variance tradeoff, as illustrated below.

\begin{figure}[!ht]
    \begin{subfigure}[b]{0.425\textwidth}
    \includegraphics[scale=0.44]{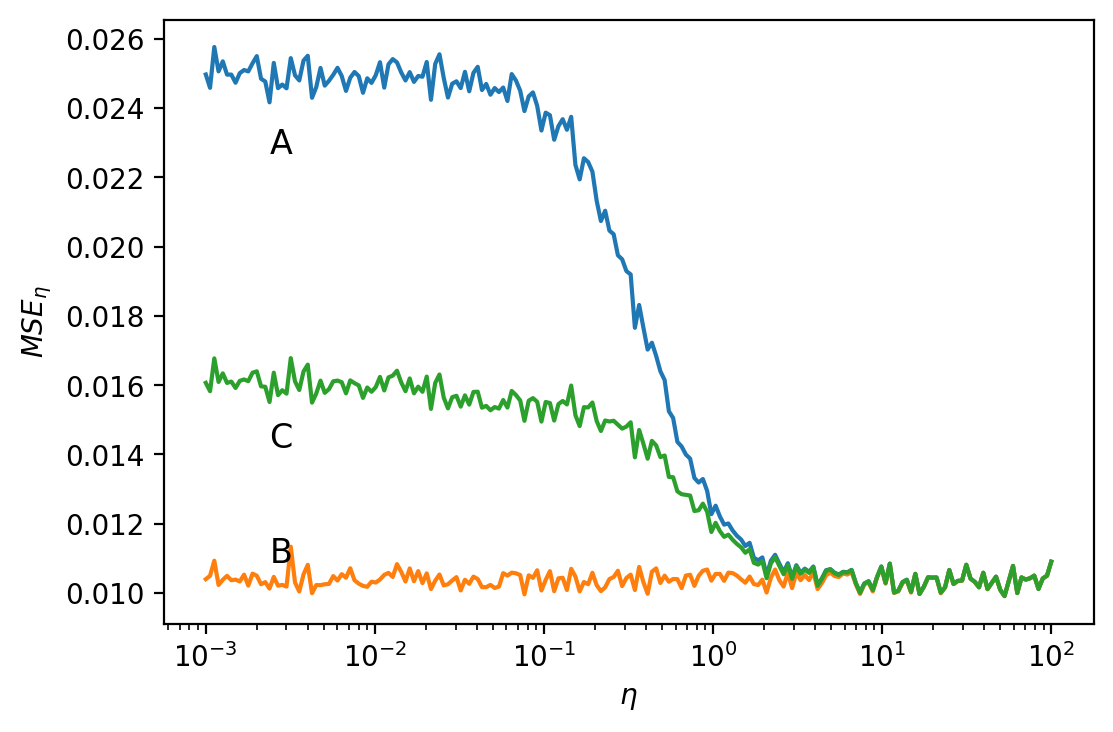}
    \captionsetup{labelformat=simple, labelsep=period}
    \caption{Scenario 1a - Low Ridge constants $C_1=C_2=0.004$.}
    \label{Biais1}
    \end{subfigure}
    \hspace{0.1\textwidth}
    \begin{subfigure}[b]{0.425\textwidth}
    \includegraphics[scale=0.44]{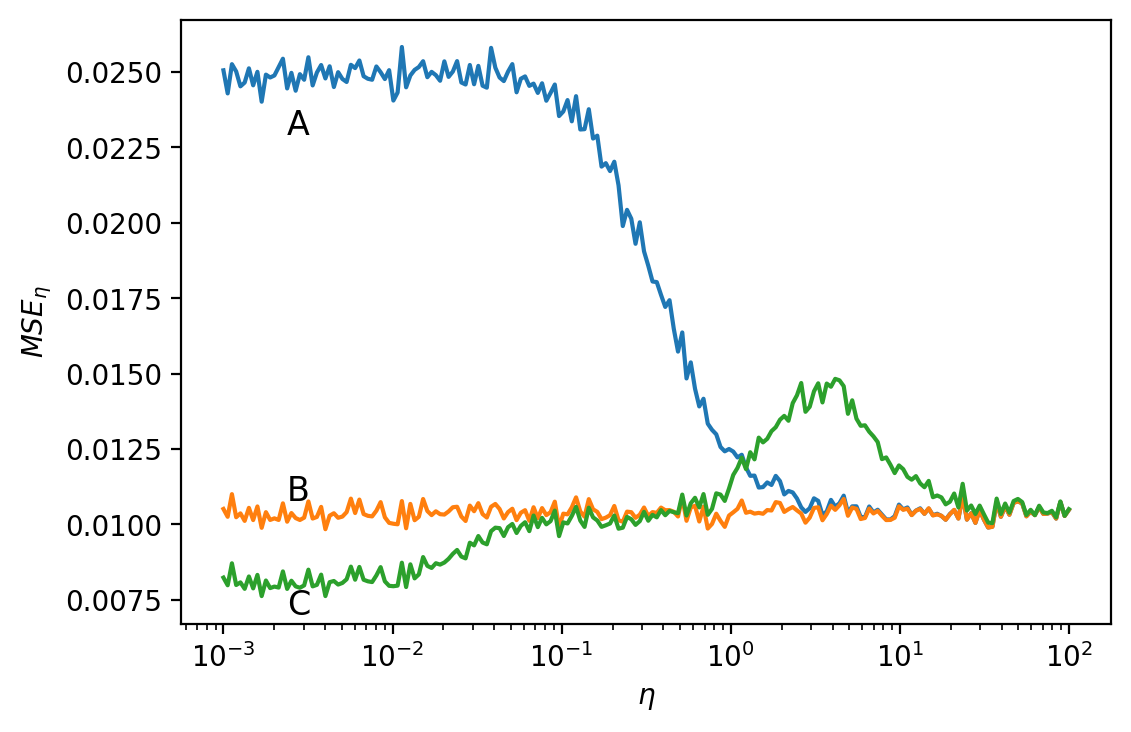}
    \captionsetup{labelformat=simple, labelsep=period}
    \caption{Scenario 1a - High Ridge constants $C_1=C_2=0.4$.}
    \label{Biais2}
    \end{subfigure}
    \caption{Quadratic risk $\mathrm{MSE}_{\eta,.}$ with respect to $\eta$ for Scenario 1a. Curves A, B C give the quadratic risk respectively for the PLS estimator \eqref{Eq:SimusPLSEstimator}, the oracle estimator \eqref{Eq:SimusAxisEstimator} and the Ridge estimator \eqref{Eq:SimusRidgeEstimator} with different choices of C1 and C2.}
    \label{BiaisSimu}
\end{figure}

On \cref{Biais1}, the Ridge regularization is low, in order to be closer to the PLS estimator. In this case, the bias induced by $\alpha_1$ and $\alpha_2$ is virtually absent, but the variance is greater and the Ridge regularization has a larger risk for small $\eta$. On \cref{Biais2}, in the opposite case where the parameters are large, the variance of the Ridge regularization is lower for small $\eta$. The bias induced by the parameters is increasing and noticeable for high $\eta.$ It can mainly be seen on the graph for $\eta$ between 1 and 10.

\subsection{Bias variance tradeoff}\label{subsec:Tradeoff}

We introduce a parameter $\nu\in[0,1]$ to represent how far $\beta$ is from given Krylov subspaces.

\paragraph{Scenario 3.}
We compute $\beta=\nu \cdot v_{1} + v_2 + (1-\nu)\cdot v_{5}$, with $\nu \in [0,1]$.  We consider $\lambda_1=3$, $\lambda_2=2$, $\lambda_3=0.06$, $\lambda_4=0.05$ and $\lambda_5=0.04$. 

The two extreme cases, $\nu=0$ and $\nu=1$, correspond to the situations where the Krylov subspace has dimension 2. The parameter $\nu$ introduces a bias from subspace $\G=\mathrm{Vect}(v_{2},v_{5})$ to $\G=\mathrm{Vect}(v_{1},v_{2})$.
The main difference between these two cases are the eigenvalues associated to each eigenvector.

We are interested at the mean square error $\mathrm{MSE}_{\nu}$ defined as
$$\mathrm{MSE}_{\nu}=\frac{1}{n \times N}\sum_{i=1}^{N}\|X(\beta_{\nu}-\hat{\beta}_{2,i})\|^2,$$
where $\hat{\beta}_{2,i}$ is the PLS estimator with 2 components according to the \(i\)\textsuperscript{th} sample. We decompose this risk into a bias term and a variance term. The bias term is $\frac{1}{n}\|X(\beta-\bar{\beta})\|^2$ with $\bar{\beta}$ defined in \eqref{Eq:barbeta}. It represents the distance between $X\beta$ and the prediction in the Krylov subspace.

\begin{figure}[ht]
    \centering
    \includegraphics[scale=0.6]{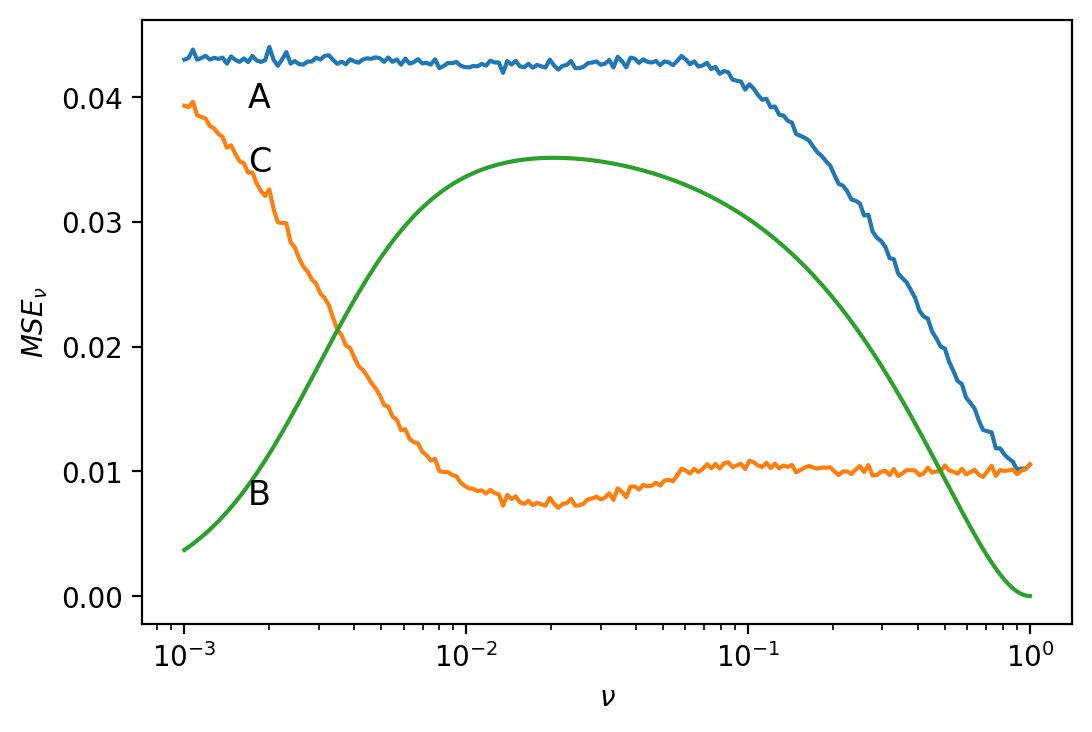}
    \caption{Bias variance tradeoff in Scenario 3. Curve A gives the quadratic risk $\mathrm{MSE}_\nu$ with respect to the parameter $\nu$. Curve B is the bias term and curve C is the difference between the risk (Curve A) and the bias term (Curve B).}
    \label{Tradeoff}
\end{figure}

\cref{Tradeoff} shows the bias variance tradeoff corresponding to \cref{Th:1}. Indeed, when $\beta$ belongs to a space of dimension 2, that is, when $\nu=0$ or $\nu=1$, the bias is minimal. The closer $\nu$ is from 0.5, the higher the distance between $\beta$ and a space of dimension 2, and the higher the bias. 

The evolution of the variance is illustrated by curve C in \cref{Tradeoff}. Our simulation shows that it decreases by changing the structure of the Krylov space. The results are similar than the ones from the previous section (Scenario 1 versus Scenario 2b), showing a smaller variance when the eigenvalues corresponding to the eigenvectors used in the construction of $\beta$ are high.

\section{Conclusion}

Our results establish non asymptotic bounds of the prediction of the PLS estimator.  Considering a non asymptotic framework, and non random covariates, allows to highlight that the procedure is efficient under a signal-to-noise condition, that is, when the PLS components are relevant enough. Moreover, our work put in evidence the influence of the Gram matrix of the covariates $\Sigma$. We adopt the Krylov space viewpoint which is a suitable framework to investigate theoretical performance. This approach enables us to apply deviation results and set prediction bounds.

To overcome the condition of sufficient signal-to-noise ratios, we propose a Ridge regularization. Based on the Krylov representation of the PLS estimator, this approach provides a similar bound than the classical PLS regression, assuming only that the Krylov components are linearly independent. This method allows us to get rid of the matrix R with the parameter $\alpha$ depending only of the dimension, the noise and the Gram matrix $\Sigma$.

Finally, a simulation study illustrates that the assumption of a sufficient signal-to-noise ratio to ensure the quality of the PLS approach makes sense. It also shows that the Ridge regularization succeeds to overcome this assumption. The importance of the eigen structure of the Gram matrix $\Sigma$ is also highlighted.

\appendix
\section{Preliminary technical results}
\label{s:technical}

\subsection[Main distribution]{Distribution properties of $\hat{\sigma}$}
This section is dedicated to some specific technical results that will be used all along the proofs.
We first state the moments and the distribution of the main quantities appearing in the construction of the PLS estimator. 

\begin{lemma}\label{lem:variances}
We have for any $i\in \NN,$
$$ \hat\sigma \sim\N_p\Bigl(\sigma,\frac{\tau^2}{n}\Sigma\Bigr) \quad \mathrm{and} \quad \Sigma^{i}\hat\sigma \sim\N_p\Bigl(\Sigma^{i}\sigma,\frac{\tau^2}{n}\Sigma^{2i+1}\Bigr).$$
In particular
$$  \EE[\hat{\sigma}^T\hat{\sigma}] =\sigma^T\sigma+\frac{\tau^2}{n}\Tr(\Sigma),  \quad \quad  \EE[\hat{\sigma}^T\Sigma^{i}\hat{\sigma}] =\sigma^T\Sigma^{i}\sigma+\frac{\tau^2}{n}\Tr(\Sigma^{i+1}),$$
and 
$$ \EE[(\hat{\sigma}-\sigma)^T\Sigma^{i}(\hat{\sigma}-\sigma)] =\frac{\tau^2}{n}\Tr(\Sigma^{i+1}).$$
\end{lemma}

The results of this lemma are a direct consequence of the definition of $\hat\sigma$ and of the fact that $\varepsilon \sim \mathcal{N}(0,\tau^2 I_n)$. The proof is thus omitted.

\subsection[Deviation inequalities]{Deviation inequalities}
\begin{proposition}\label[prop]{Lemme:ExtensionLaurent}
Let $U\sim\N_D(m,tA)$ with $D\in \mathbb{N}$, $m \in \mathbb{R}^D$, $t\in \RR_+$ and $A\in \mathbb{R}^{D\times D}$ a symmetric positive matrix. Define, for any $s\in \mathbb{N}$, 
$$\Xi_s=t^2\Tr(A^{2(s+1)}) +2t\rho(A^{s+1})\| A^{\frac{s}{2}}m\|^2,$$ 
Then, for all $s\in \mathbb{N}$ and $x\geq 0$,
\begin{align*}
    i) & \quad \PP\bigg(U^TA^sU-\EE[U^TA^sU]\ge 2\sqrt{\Xi_s x}+2t\rho(A)^{s+1}x\bigg) \le e^{-x},\\
    ii) &  \quad \PP\bigg(U^TA^sU-\EE[U^TA^sU]\le -2\sqrt{\Xi_s x}\bigg) \le e^{-x}.
\end{align*}
\end{proposition}
\begin{proof}
The result follows from an application of Lemma 2 from \textcite{Laurent}. For more details see Proposition 2 in \textcite{Cast}.
\end{proof}
Before stating additional results, we introduce, for any $x\in \mathbb{R}^+$ and $i\in \{0,...,2K-1\}$, the following quantities:
\begin{align}
\label{eqn:T2}	 \mathbf{T}_{1,i}(x)&=g(x)\frac{\tau^2}{n}\Tr(\Sigma^{i+1})+2\sqrt{2}\sqrt{\frac{\tau^2}{n}}\rho(\Sigma)^{\frac{i+1}{2}}\sqrt{x}\|\Sigma^{\frac{i}{2}}\sigma\|,\\
\label{eqn:T3}	\mathbf{T}_{2,i}(x)& =g(x)\frac{\tau^2}{n}\Tr(\Sigma^{i+1}),
\end{align}
with 
\begin{equation}
\label{eqn:C}
g(x)=1+2x+2\sqrt{x}.
\end{equation}

The following proposition will be the main element on which our proof is based. It provides deviation results on the main quantities of interest.

\begin{proposition}\label{Prop:Majoration3Termes}
For any $0<\delta<1$, let $(\mathcal{A}_{i,\delta})_{i=0}^{2K-1}, (\mathcal{B}_{i,\delta})_{i=0}^{2K-1}$ the events respectively defined as
\begin{align*}
%\label{eq12} 
\mathcal{A}_{i,\delta} &= \left\lbrace	 \left|  \hat{\sigma}^T\Sigma^{i}\hat{\sigma} - \sigma^T\Sigma^{i}\sigma \right| \le \mathbf{T}_{1,i}(x_{\delta})\right\rbrace,\\
%\label{eq13} 
\mathrm{and} \quad \mathcal{B}_{i,\delta} &= \left\lbrace		(\hat{\sigma}-\sigma)^T\Sigma^{i}(\hat{\sigma}-\sigma)\le \mathbf{T}_{2,i}(x_{\delta}) \right\rbrace,
\end{align*}
with $x_{\delta}=\ln(6K/\delta)$. Then, 
$$ \mathbb{P}(\mathcal{A}_\delta) \geq 1-\delta \quad \mathrm{where} \quad \mathcal{A}_\delta:= \bigcap_{i=0}^{2K-1} \A_{i,\delta}\cap \B_{i,\delta}.$$ 
\end{proposition}
\begin{proof}
First, applying item i) and ii) of \cref{Lemme:ExtensionLaurent} on the variable $\hat\sigma$ with $s=i$, $t=\frac{\tau^2}{n}$, $m=\sigma$ and $A=\Sigma$, we get, for all $x\in \mathbb{R}_+$
$$ \mathbb{P}( | \hat\sigma^T \Sigma^i \hat\sigma - \mathbb{E}[\hat\sigma^T \Sigma^i \hat\sigma] | \geq B_{i,x}) \leq e^{-x},$$
where 
\begin{align*}
B_{i,x} 
& =  2\sqrt{x}\sqrt{\left( \frac{\tau^2}{n}\right)^2 \Tr(\Sigma^{2(i+1)}) + 2\frac{\tau^2}{n} \rho(\Sigma^{i+1}) \| \Sigma^{i/2}\sigma\|^2} + 2 \frac{\tau^2}{n} \rho(\Sigma)^{i+1} x\\
& \leq  2\sqrt{x} \frac{\tau^2}{n} \sqrt{\Tr(\Sigma^{2(i+1)})} + 2x\frac{\tau^2}{n} \rho(\Sigma)^{i+1} + 2\sqrt{2x} \sqrt{\frac{\tau^2}{n}} \sqrt{\rho(\Sigma^{i+1})} \sqrt{\sigma^T\Sigma^i \sigma}\\
& \leq (2\sqrt{x}+2x) \frac{\tau^2}{n} \Tr(\Sigma^{i+1}) + 2\sqrt{2x} \sqrt{\frac{\tau^2}{n}}\rho(\Sigma)^{\frac{i+1}{2}} \| \Sigma^{i/2}\sigma\|.
\end{align*}
\cref{lem:variances} then allow to obtain 
$$\PP(\A_{i,\delta}^{C})\le 2 \frac{\delta}{6K}.$$
Using again i) of \cref{Lemme:ExtensionLaurent} on the variable $\hat{\sigma}-\sigma$ with $s=i$, $t=\frac{\tau^2}{n}$, $m=0$ and $A=\Sigma$, we get $\PP(\B_{i,\delta}^{C})\le \frac{\delta}{6K}$.
Using the union bound we have $\PP(\A^{C})\le 2K (\frac{\delta}{3K})+ 2K (\frac{\delta}{6K})\le \delta.$
\end{proof}

We can re-formulate \cref{Prop:Majoration3Termes} above as follows. 
\begin{corollary}\label{Cor:Majoration3Termes}
Let $0<\delta<1$. Denote $C_\delta=\max(g(x_\delta),\,2\sqrt{2 x_\delta})$. Then, on the set $\mathcal A_\delta$, for all $i \in \lbrace 0,\dots,p \rbrace$, 
\begin{align}
\label{eq12} \left|  \hat{\sigma}^T\Sigma^{i}\hat{\sigma} - \sigma^T\Sigma^{i}\sigma \right| & \le C_\delta\left(\frac{\tau^2}{n}\Tr(\Sigma^{i+1})+\sqrt{\frac{\tau^2}{n}}\rho(\Sigma)^{\frac{i+1}{2}}\|\Sigma^{\frac{i}{2}}\sigma\|\right),\\
\label{eq13} \mathrm{and} \quad (\hat{\sigma}-\sigma)^T\Sigma^{i}(\hat{\sigma}-\sigma) & \le C_{\delta}\frac{\tau^2}{n}\Tr(\Sigma^{i+1}).
\end{align}
\end{corollary}
The proof is a direct consequence of \cref{Prop:Majoration3Termes} and is thus omitted.

\subsection{Inversion of the estimated correlation matrix}
We first state the inversion of the matrix
\begin{equation}
\hat{R}:=D^{-\frac{1}{2}}\hat{\Theta}D^{-\frac{1}{2}}
\label{eq:hatR}
\end{equation}
with high probability. This matrix will play a central role in the proof displayed in Section \ref{s:proof}. Observe that we consider here the matrix $D$ and not its estimation.
\begin{lemma}\label{lem:R}
Suppose \ref{ass:inverse} and \ref{ass:signal} hold. Then, on the event $\mathcal{A}_\delta$ defined in  \cref{Prop:Majoration3Termes}, we have
\begin{align*}
    \rho_{\min}(\hat{R})\ge \frac{\rho_{\min}(R)}{2} \quad \mathrm{and} \quad \rho(\hat{R}-R)\le\rho(R).
\end{align*}
\end{lemma}
\begin{proof}
Let $x\in \RR^{K}$ such that $x^Tx=1$. Then 
\begin{align*}
	x^T\hat{R}x&=x^TD^{-\frac{1}{2}}\hat{G}^T\Sigma\hat{G}D^{-\frac{1}{2}}x\\
	&=x^TD^{-\frac{1}{2}}(\hat{G}-G)^T\Sigma(\hat{G}-G)D^{-\frac{1}{2}}x+x^TRx+2x^TD^{-\frac{1}{2}}(\hat{G}-G)^{T}\Sigma G D^{-\frac{1}{2}}x\\
	&\ge x^TRx-2|x^TD^{-\frac{1}{2}}(\hat{G}-G)\Sigma GD^{-\frac{1}{2}}x|.
\end{align*}
Applying inequality $2ab\le a^2+b^2$ with well chosen $a,b$,
\begin{align}
\nonumber	x^T\hat{R}x&\ge x^TRx-\frac{1}{4}x^TRx-4x^TD^{-\frac{1}{2}}(\hat{G}-G)^T\Sigma(\hat{G}-G)D^{-\frac{1}{2}}x\\
\label{eq:ineqRhat}	&\ge \frac{3}{4}\rho_{\min}(R)-4\rho(D^{-\frac{1}{2}}(\hat{G}-G)^T\Sigma(\hat{G}-G)D^{-\frac{1}{2}}).
\end{align}
We now seek for an upper bound of $\rho(D^{-\frac{1}{2}}(\hat{G}-G)^T\Sigma(\hat{G}-G)D^{-\frac{1}{2}})$. We use the classic inequality 
\begin{equation}
\rho(A)\le K\underset{1\le i,j\le K}{\max}|A_{ij}|,
\label{eq:inequA}
\end{equation}
for any positive semi-definite matrix $A \in \mathbb{R}^{K\times K}$. In our setting, this writes as 
$$\rho(D^{-\frac{1}{2}}(\hat{G}-G)^T\Sigma(\hat{G}-G)D^{-\frac{1}{2}})\le K \underset{1\le i,j\le K}{\max}\frac{(\hat{\sigma}-\sigma)^T\Sigma^{i+j-1}(\hat{\sigma}-\sigma)}{\sqrt{\sigma^T\Sigma^{2i-1}\sigma}\sqrt{\sigma^T\Sigma^{2j-1}\sigma}}.$$
Applying successively \cref{Cor:Majoration3Termes} and Cauchy-Schwarz inequality, we obtain that, on the set $\mathcal{A}_\delta$,
\begin{align*}
    \rho(D^{-\frac{1}{2}}(\hat{G}-G)^T\Sigma(\hat{G}-G)D^{-\frac{1}{2}})&\le \underset{1\le i,j\le K}{\max}\frac{C_{\delta} K\frac{\tau^2}{n}\Tr(\Sigma^{i+j})}{\sqrt{\sigma^T\Sigma^{2i-1}\sigma}\sqrt{\sigma^T\Sigma^{2j-1}\sigma}}\\
    &\le C_{\delta}\underset{1\le i,j\le K}{\max}\frac{\sqrt{K\frac{\tau^2}{n}\Tr(\Sigma^{2i})}}{\sqrt{\sigma^T\Sigma^{2i-1}\sigma}}\frac{\sqrt{K\frac{\tau^2}{n}\Tr(\Sigma^{2j})}}{\sqrt{\sigma^T\Sigma^{2j-1}\sigma}}\\
    &\le \frac{C_{\delta}}{t_{\delta,R}},
\end{align*}
where the constant $t_{\delta,R}$ is defined in \ref{ass:signal}.
Hence, \eqref{eq:ineqRhat} becomes $$ 
x^T\hat{R}x\geq \frac{3}{4}\rho_{\min}(R)-4 \frac{C_{\delta}}{t_{\delta,R}}.$$
Since $t_{\delta,R}\geq  \frac{16\,C_\delta}{\rho_{\min}(R)}$, it follows that $\rho_{\min}(\hat{R})\ge \frac{\rho_{\min}(R)}{2}$.

Let us now prove the second inequality of \cref{lem:R}. We have
\begin{align*}
    x^T(\hat{R}-R)x&=x^TD^{-\frac{1}{2}}(\hat{\Theta}-\Theta)D^{-\frac{1}{2}}x\\
	&=x^TD^{-\frac{1}{2}}(\hat{G}^T\Sigma\hat{G}-G^T\Sigma G)D^{-\frac{1}{2}}x\\
	&=x^TD^{-\frac{1}{2}}\big((\hat{G}-G)^T\Sigma(\hat{G}-G)+2G^T\Sigma(\hat{G}-G)\big)D^{-\frac{1}{2}}x\\
	&\le x^TD^{-\frac{1}{2}}(\hat{G}-G)^T\Sigma(\hat{G}-G)D^{-\frac{1}{2}}x+2|x^TD^{-\frac{1}{2}}G^T\Sigma(\hat{G}-G)D^{-\frac{1}{2}}x|.
	\end{align*}
Using again inequality $2ab\leq a^2+b^2$ for any real $a,b$,
\begin{align*}
    x^T(\hat{R}-R)x	
	&\le \rho(D^{-\frac{1}{2}}(\hat{G}-G)^T\Sigma(\hat{G}-G)D^{-\frac{1}{2}})+\frac{1}{2}x^TRx + 2x^TD^{-\frac{1}{2}}(\hat{G}-G)^T\Sigma(\hat{G}-G)D^{-\frac{1}{2}}x\\
	&\le 3\rho(D^{-\frac{1}{2}}(\hat{G}-G)^T\Sigma(\hat{G}-G)D^{-\frac{1}{2}})+\frac{1}{2}\rho(R).
\end{align*}
Hence, $ x^T(\hat{R}-R)x\le 3\frac{C_{\delta}}{t_{\delta,R}}+\frac{1}{2}\rho(R)$ on the set $\mathcal{A}_\delta$. We deduce that $\rho(\hat{R}-R)\le\rho(R)$ since $t_{\delta,R}\geq  \frac{16\,C_\delta}{\rho_{\min}(R)}$.
\end{proof}

\subsection{Upper bound on three different terms}
Recall that $\Lambda=\Theta^{-1}G^T\sigma$.

\begin{lemma} \label{lem:termI}
On the event $\mathcal{A}_\delta$ defined in  \cref{Prop:Majoration3Termes}, 
$$ \mathrm{I}:=\Lambda^T(G-\hat{G})^T\Sigma(G-\hat{G})\Lambda \le C_{\delta}\frac{\tau^2}{n}\left(\sum_{i=1}^{K}|\Lambda_i|\sqrt{\Tr(\Sigma^{2i})}\right)^2.$$
\end{lemma}
\begin{proof}
First, we can remark that  
$$\I=\sum_{i,j=1}^{K}\Lambda_i\Lambda_j(\hat{\sigma}-\sigma)\Sigma^{i+j-1}(\hat{\sigma}-\sigma).$$
Then, \cref{Cor:Majoration3Termes} states that, on the event $\mathcal{A}_\delta$,
$$\I\le C_{\delta}\frac{\tau^2}{n}\sum_{i,j=1}^{K}|\Lambda_i| |\Lambda_j|\Tr(\Sigma^{i+j}).$$ Using Cauchy-Schwarz inequality, we obtain
\begin{align*}
	\I 	&\le C_{\delta}\frac{\tau^2}{n}\sum_{i,j=1}^{K}|\Lambda_i| |\Lambda_j| \sqrt{\Tr(\Sigma^{2i})}\sqrt{\Tr(\Sigma^{2j})}
	= C_{\delta}\frac{\tau^2}{n}\left(\sum_{i=1}^{K}|\Lambda_i|\sqrt{\Tr(\Sigma^{2i})}\right)^2.
\end{align*}
\end{proof}

\begin{lemma} \label{lem:termII}
Suppose \ref{ass:signal} is satisfied.
On the event $\mathcal{A}_\delta$ defined in \cref{Prop:Majoration3Termes}, 
$$\mathrm{II}:=\Lambda^T(\hat{\Theta}-\Theta)D^{-1}(\hat{\Theta}-\Theta)\Lambda \le 2\frac{\tau^2}{n}\frac{C_{\delta}^2}{t_{\delta,R}}\left(\sum_{j=1}^{K}\sqrt{\Tr(\Sigma^{2j})}\lvert\Lambda_j\rvert\right)^2+2\frac{\tau^2}{n}C_{\delta}^2\rho(R)\|\tilde{\Lambda}\|^2\sum_{j=1}^{K}\frac{\rho(\Sigma)^{2j}}{\sigma^T\Sigma^{2j-1}\sigma},$$
where $\tilde{\Lambda}=D^{\frac{1}{2}}\Lambda=(\sqrt{\sigma^T \Sigma^{2l-1}\sigma}\times \Lambda_l)_{l=1..K}$.
\end{lemma}
\begin{proof}
First note that
\begin{align*}
\II & =\sum_{k=1}^{K}\frac{1}{\sigma^T\Sigma^{2k-1}\sigma}\left(\sum_{j=1}^{K}(\hat{\Theta}_{kj}-\Theta_{kj})\Lambda_{j}\right)^2\\ 
& \leq \sum_{k=1}^{K}\frac{1}{\sigma^T\Sigma^{2k-1}\sigma}\left(\sum_{j=1}^{K}|\hat\sigma^T \Sigma^{k+j-1}\hat\sigma - \sigma^T \Sigma^{k+j-1}\sigma|\times |\Lambda_{j}|\right)^2.
\end{align*}
With \cref{Cor:Majoration3Termes}, on the event $\mathcal{A}_\delta$,
$$ \II \le\sum_{k=1}^{K}\frac{1}{\sigma^T\Sigma^{2k-1}\sigma}\left(\sum_{j=1}^{K}\left(C_{\delta}\frac{\tau^2}{n}\Tr(\Sigma^{j+k})+C_{\delta}\sqrt{\frac{\tau^2}{n}}\rho(\Sigma)^{\frac{j+k}{2}}\sqrt{\sigma^T\Sigma^{j+k-1}\sigma}\right)|\Lambda_{j}|\right)^2.$$
Then,
\begin{align*}
    \II 
    &\le2\sum_{k=1}^{K}\frac{1}{\sigma^T\Sigma^{2k-1}\sigma}\left(C_{\delta}\sum_{j=1}^{K}\frac{\tau^2}{n}\Tr(\Sigma^{j+k})|\Lambda_j|\right)^2\\
    & \quad +2\sum_{k=1}^{K}\frac{1}{\sigma^T\Sigma^{2k-1}\sigma}\left(C_{\delta}\sum_{j=1}^{K}\sqrt{\frac{\tau^2}{n}}\rho(\Sigma)^{\frac{j+k}{2}}\sqrt{\sigma^T\Sigma^{j+k-1}\sigma}|\Lambda_j|\right)^2\\
    &\le 2C_{\delta}^{2}\sum_{k=1}^{K}\frac{\frac{\tau^2}{n}\Tr(\Sigma^{2k})}{\sigma^T\Sigma^{2k-1}\sigma}\left(\sum_{j=1}^{K}\sqrt{\frac{\tau^2}{n}}\sqrt{\Tr(\Sigma^{2j})}|\Lambda_j|\right)^2\\
    &\quad +2\frac{\tau^2}{n}C_{\delta}^{2}\sum_{k=1}^{K}\frac{1}{\sigma^T\Sigma^{2k-1}\sigma}\left(\sum_{j=1}^{K}\rho(\Sigma)^{\frac{j+k}{2}}\frac{\sqrt{\sigma^T\Sigma^{j+k-1}\sigma}}{\sqrt{\sigma^T\Sigma^{2j-1}\sigma}}\cdot\sqrt{\sigma^T\Sigma^{2j-1}\sigma}|\Lambda_j| \right)^2.
\end{align*}
Using \ref{ass:signal} for the first term and Cauchy-Schwarz inequality for the second term, we get
\begin{align*}
   \II &\le 2C_{\delta}^2\sum_{k=1}^{K}\frac{1}{K t_{\delta,R}}\frac{\tau^2}{n}\left(\sum_{j=1}^{K}\sqrt{\Tr(\Sigma^{2j})}\Lambda_j\right)^2 \\
    & \quad +2\frac{\tau^2}{n}C_{\delta}^{2}\sum_{k=1}^{K}\frac{1}{\sigma^T\Sigma^{2k-1}\sigma}\sum_{j=1}^{K}\rho(\Sigma)^{j+k}\frac{\sigma^T\Sigma^{j+k-1}\sigma}{\sigma^T\Sigma^{2j-1}\sigma}\sum_{l=1}^{K}\sigma^T\Sigma^{2l-1}\sigma\,\Lambda_l^2\\
    &\le 2\frac{C_{\delta}^2}{t_{\delta,R}}\frac{\tau^2}{n}\left(\sum_{j=1}^{K}\sqrt{\Tr(\Sigma^{2j})}|\Lambda_j|\right)^2 \\
    & \quad+2\frac{\tau^2}{n}C_{\delta}^{2}\sum_{j,k=1}^{K}\frac{\rho(\Sigma)^{j}}{\sqrt{\sigma^T\Sigma^{2j-1}\sigma}}\frac{\sigma^T\Sigma^{j+k-1}\sigma}{\sqrt{\sigma^T\Sigma^{2k-1}\sigma}\sqrt{\sigma^T\Sigma^{2j-1}\sigma}}\frac{\rho(\Sigma)^{k}}{\sqrt{\sigma^T\Sigma^{2k-1}\sigma}}\|\tilde{\Lambda}\|^2.
\end{align*}
Let $v\in\RR^K$ the vector defined as $v_j=\frac{\rho(\Sigma)^j}{\sqrt{\sigma^T\Sigma^{2j-1}\sigma}}$ for all $j\in\lbrace 1,\dots, K \rbrace$. By definition of the matrix $R$ (see \eqref{eq:Rmatrix}), the second term on the right-hand side writes as
$$2\frac{\tau^2}{n}C_{\delta}^{2}v^TRv\|\tilde{\Lambda}\|^2\\
    \le 2\frac{\tau^2}{n}C_{\delta}^2\rho(R)\|v\|^2\|\tilde{\Lambda}\|^2.$$ \cref{lem:termII} follows.
\end{proof}

Denote \begin{equation}\label{eq:defLbar}
    \bar\Lambda=D^{-1}G^T\sigma = \left( \frac{\sigma^T\Sigma^{i-1}\sigma}{\sigma^T\Sigma^{2i-1}\sigma}\right)_{i=1\dots K},
\end{equation} the norm of the marginal projection. That is, the projection of $Y$ on the $i^\text{th}$ vector of the Krylov space, with a normalized vector, without projecting on other dimensions. If $R$ is ill-conditioned, $\Lambda$ and $\bar\Lambda$ differ much, and one cannot retrieve the information on $\Lambda$ from $\bar \Lambda$.

\begin{lemma} \label{lem:termIII}
Suppose that \ref{ass:inverse} and \ref{ass:signal} hold. Then, on the event $\mathcal{A}_\delta$ defined in  \cref{Prop:Majoration3Termes}, 
\begin{align*}
\mathrm{III} &:=(\hat{G}^T\hat{\sigma}-G^T\sigma)^T\Theta^{-1}(\hat{G}^T\hat{\sigma}-G^T\sigma) \\ 
& \le 2\frac{C_{\delta}^2}{\rho_{\min}(R)}\frac{1}{K\,t_{\delta,R}}\frac{\tau^2}{n}\sum_{i=1}^{K}\Tr(\Sigma^{i})\bar\Lambda_i +2\frac{C_{\delta}^2}{\rho_{\min}(R)}\frac{\tau^2}{n}\sum_{i=1}^{K}\rho(\Sigma)^{i}\bar{\Lambda_{i}}.
\end{align*}
\end{lemma}
\begin{proof}
First, \cref{Cor:Majoration3Termes} gives
\begin{align*}
 	\III	&=
 	(\hat{\sigma}^T\hat{G}-\sigma^TG)^TD^{-\frac{1}{2}}R^{-1}D^{-\frac{1}{2}}(\hat{G}^T\hat{\sigma}-G^T\sigma)\\
 	&\le\rho_{\min}(R)^{-1}(\hat{G}^T\hat{\sigma}-G^T\sigma)^TD^{-1}(\hat{G}^T\hat{\sigma}-G^T\sigma)\\
 	&\le\rho_{\min}(R)^{-1}\sum_{i=1}^{K}\frac{(\hat{\sigma}^T\Sigma^{i-1}\hat{\sigma}-\sigma^T\Sigma^{i-1}\sigma)^2}{\sigma^T\Sigma^{2i-1}\sigma}\\
 & \le\rho_{\min}(R)^{-1}C_\delta^{2}\sum_{i=1}^{K}\frac{\bigg(\frac{\tau^2}{n}\Tr(\Sigma^{i})
 	+\sqrt{\frac{\tau^2}{n}}\rho(\Sigma)^{\frac{i}{2}}\sqrt{\sigma^T\Sigma^{i-1}\sigma}\bigg)^2}{\sigma^T\Sigma^{2i-1}\sigma}.
 \end{align*}
Then, with \ref{ass:signal},
\begin{align*}
 	\III	
 	&\le \rho_{\min}(R)^{-1}C_\delta^{2}\,2\sum_{i=1}^{K}\Bigg(\left(\frac{\tau^2}{n}\right)^2\frac{\Tr(\Sigma^i)^2}{\sigma^T\Sigma^{2i-1}\sigma}+\frac{\tau^2}{n}\frac{\rho(\Sigma)^i\sigma^T\Sigma^{i-1}\sigma}{\sigma^T\Sigma^{2i-1}\sigma}\Bigg)\\
	&\le \rho_{\min}(R)^{-1}C_\delta^{2}\,2\sum_{i=1}^{K}\Bigg(\frac{\tau^2}{n}\Tr(\Sigma^i)\frac{\frac{\tau^2}{n}\Tr(\Sigma^{i})}{Kt_{\delta,R}\frac{\tau^2}{n}\rho(\Sigma)^{i}\Tr(\Sigma^{i})}+\frac{\tau^2}{n}\frac{\rho(\Sigma)^i\sigma^T\Sigma^{i-1}\sigma}{\sigma^T\Sigma^{2i-1}\sigma}\Bigg).
\end{align*}
Using the inequality 
\begin{equation}\label{eq:ratioS} \frac{1}{\rho(\Sigma)^i}\leq \frac{\sigma^T\Sigma^{i-1}\sigma}{\sigma^T\Sigma^{2i-1}\sigma} \quad \forall i \in \lbrace 1,\dots, p \rbrace,\end{equation}
we obtain
$$
 	\III	\le \rho_{\min}(R)^{-1}C_\delta^{2}\,2\frac{\tau^2}{n}\sum_{i=1}^{K}\Bigl(\Tr(\Sigma^i)\frac{1 }{K\,t_{\delta,R}}+\rho(\Sigma)^{i}\Bigr)\frac{\sigma^T\Sigma^{i-1}\sigma}{\sigma^T\Sigma^{2i-1}\sigma}.$$
\cref{lem:termIII} follows with the definition of $\bar{\Lambda}$ in \eqref{eq:defLbar}.	
\end{proof}

\section[Proof of Theorem 3.1]{Proof of \cref{Th:1}}
\label{s:proof}

First, we can remark that
\begin{align}
\frac{1}{n} \| X(\hat\beta_K - \beta)\|^2
& \leq \frac{2}{n} \| X(\bar\beta - \beta)\|^2 + \frac{2}{n} \| X(\hat \beta_K - \bar\beta)\|^2, \nonumber \\
& =  \frac{2}{n} \inf_{v\in [G]} \| X(\beta-v)\|^2 +\frac{2}{n} \|X(\hat\beta_K-\bar\beta)\|^2,
\label{eq:beginproof}
\end{align}
where $\bar\beta$ has been introduced in \eqref{Eq:barbeta}. Then, we use the following decomposition:
\begin{align*}
	\hat\beta_K - \bar{\beta} 
	& = \hat{G}\hat{\Theta}^{-1}\hat{G}^T\hat{\sigma}-G\Theta^{-1}G^T\sigma \\
	&=(\hat{G}-G)\hat{\Theta}^{-1}\hat{G}^T\hat{\sigma}+G(\hat{\Theta}^{-1}-\Theta^{-1})\hat{G}^T\hat{\sigma}+G\Theta^{-1}(\hat{G}^T\hat{\sigma}-G^T\sigma).
\end{align*}
It yields
\begin{align}
\lefteqn{\frac{1}{n} \|X(\hat\beta_K-\bar\beta)\|^2} \nonumber \\
    \nonumber& = (\hat{\beta}_{K}-\bar\beta)^T\Sigma(\hat{\beta}_{K}-\bar\beta)\\
    \nonumber&\le 4\hat{\sigma}^T\hat{G}\hat{\Theta}^{-1}(\hat{G}-G)^T\Sigma(\hat{G}-G)\hat{\Theta}^{-1}\hat{G}^T\hat{\sigma} + 4 \hat{\sigma}^T\hat{G}(\hat{\Theta}^{-1}-\Theta^{-1})\Theta(\hat{\Theta}^{-1}-\Theta^{-1})\hat{G}^T\hat{\sigma}\\
    \nonumber&\quad +2(\hat{\sigma}^T\hat{G}-\sigma^TG)\Theta^{-1}(\hat{G}^T\hat{\sigma}-G^T\sigma)\\
    & = 4\underbrace{\hat{\sigma}^T\hat{G}\hat{\Theta}^{-1}(\hat{G}-G)^T\Sigma(\hat{G}-G)\hat{\Theta}^{-1}\hat{G}^T\hat{\sigma}}_{:=T_1}+4\underbrace{\hat{\sigma}^T\hat{G}(\hat{\Theta}^{-1}-\Theta^{-1})\Theta(\hat{\Theta}^{-1}-\Theta^{-1})\hat{G}^T\hat{\sigma}}_{:=T_2}+2\,\III.
    \label{ineqPrincipale}
\end{align}
In the following, we shall bound these terms using $\I$, $\II$ and $\III$ defined respectively in \cref{lem:termI}, \cref{lem:termII}, and \cref{lem:termIII}.

\subsection[Bound on the first term]{Bound on $T_1$}
We decompose the term $T_1$ as follows: 
\begin{align*}
   T_1&=\hat{\sigma}^T\hat{G}\hat{\Theta}^{-1}(\hat{G}-G)^T\Sigma(\hat{G}-G)\hat{\Theta}^{-1}\hat{G}^T\hat{\sigma}\\
    &\le 2\hat{\sigma}^T\hat{G}\Theta^{-1}(\hat{G}-G)^T\Sigma(\hat{G}-G)\Theta^{-1}\hat{G}^T\hat{\sigma}\\
    &\quad +2\hat{\sigma}^T\hat{G}(\hat{\Theta}^{-1}-\Theta^{-1})(\hat{G}-G)\Sigma(\hat{G}-G)(\hat{\Theta}^{-1}-\Theta^{-1})\hat{G}^T\hat{\sigma}\\
    &=:T_{11}+T_{12}.
\end{align*}

\paragraph{Control of the term $T_{12}$.}

First, remark that
\begin{align*}
   T_{12}&=2\,\hat{\sigma}^T\hat{G}(\hat{\Theta}^{-1}-\Theta^{-1})(\hat{G}-G)^T\Sigma(\hat{G}-G)(\hat{\Theta}^{-1}-\Theta^{-1})\hat{G}^T\hat{\sigma}\\
    &= 2\,\hat{\sigma}^T\hat{G}\Theta^{-1}(\hat{\Theta}-\Theta)\hat{\Theta}^{-1}(\hat{G}-G)^{T}\Sigma(\hat{G}-G)\hat\Theta^{-1}(\hat{\Theta}-\Theta)\Theta^{-1}\hat{G}^T\hat{\sigma}\\
    & = 2\,\hat{\sigma}^T\hat{G}\Theta^{-1}(\hat{\Theta}-\Theta)D^{-\frac{1}{2}}\hat{R}^{-1}\bigg(D^{-\frac{1}{2}}(\hat{G}-G)^T\Sigma(\hat{G}-G)D^{-\frac{1}{2}}\bigg)\hat{R}^{-1}D^{-\frac{1}{2}}(\hat{\Theta}-\Theta)\Theta^{-1}\hat{G}^T\hat{\sigma}\\
    &\le\frac{8}{\rho_{\min}(R)^2}\rho\bigg(D^{-\frac{1}{2}}(\hat{G}-G)^T\Sigma(\hat{G}-G))D^{-\frac{1}{2}}\bigg)\hat{\sigma}^T\hat{G}\Theta^{-1}(\hat{\Theta}-\Theta)D^{-1}(\hat{\Theta}-\Theta)\Theta^{-1}\hat{G}^T\hat{\sigma},
\end{align*}
where we have used \cref{lem:R}. Using inequality \eqref{eq:inequA} and \cref{Cor:Majoration3Termes}, we deduce that, on the event $\mathcal{A}_\delta$,
\begin{align*}
T_{12} & \le\frac{8}{\rho_{\min}(R)^2}K\underset{i,j}{\max}\left\lbrace\frac{(\hat\sigma-\sigma)^T\Sigma^{i+j-1} (\hat\sigma - \sigma)}{\sqrt{\sigma^T\Sigma^{2i-1}\sigma}\sqrt{\sigma^T\Sigma^{2j-1}\sigma}}\right\rbrace\hat{\sigma}^T\hat{G}\Theta^{-1}(\hat{\Theta}-\Theta)D^{-1}(\hat{\Theta}-\Theta)\Theta^{-1}\hat{G}^T\hat{\sigma}\\
& \le\frac{8C_{\delta}}{\rho_{\min}(R)^2}K\underset{i,j}{\max}\left\lbrace\frac{\frac{\tau^2}{n}\Tr(\Sigma^{i+j})}{\sqrt{\sigma^T\Sigma^{2i-1}\sigma}\sqrt{\sigma^T\Sigma^{2j-1}\sigma}}\right\rbrace\hat{\sigma}^T\hat{G}\Theta^{-1}(\hat{\Theta}-\Theta)D^{-1}(\hat{\Theta}-\Theta)\Theta^{-1}\hat{G}^T\hat{\sigma}.
\end{align*}
Since $\Tr(\Sigma^{i+j})\le \sqrt{\Tr(\Sigma^{2i})}\sqrt{\Tr(\Sigma^{2j})}$ for any $i,j\in \lbrace 1,\dots, K \rbrace$, \ref{ass:signal} gives
\begin{align*}
T_{12}     
    & \le \frac{8C_{\delta}}{\rho_{\min}(R)^2}K\cdot\frac{1}{K\,t_{\delta,R}} \hat{\sigma}^T\hat{G}\Theta^{-1}(\hat{\Theta}-\Theta)D^{-1}(\hat{\Theta}-\Theta)\Theta^{-1}\hat{G}^T\hat{\sigma}\\
       &\le \frac{16C_{\delta}}{\rho_{\min}(R)^2 t_{\delta,R}}(\hat{\sigma}^T\hat{G}-\sigma^TG)\Theta^{-1}(\hat{\Theta}-\Theta)D^{-1}(\hat{\Theta}-\Theta)\Theta^{-1}(\hat{G}^T\hat{\sigma}-G^T\sigma)\\
    &\quad+\frac{16C_{\delta}}{\rho_{\min}(R)^2 t_{\delta,R}}\sigma^TG\Theta^{-1}(\hat{\Theta}-\Theta)D^{-1}(\hat{\Theta}-\Theta)\Theta^{-1}G^T\sigma.
\end{align*}
Recall that $R=D^{-1/2}\Theta D^{-1/2}$ and $\hat R=D^{-1/2}\hat \Theta D^{-1/2}$. \cref{lem:R} implies that, on the event $\mathcal{A}_\delta$,    
 \begin{align*}
   T_{12}  
    &\le \frac{16C_{\delta}}{\rho_{\min}(R)^2 t_{\delta,R}}(\hat{\sigma}^T\hat{G}-\sigma^TG)D^{-\frac{1}{2}}R^{-1}(\hat{R}-R)^{2}R^{-1}D^{-\frac{1}{2}}(\hat{G}^T\hat{\sigma}-G^T\sigma)\\  
    &\quad+\frac{16C_{\delta}}{\rho_{\min}(R)^2 t_{\delta,R}}\sigma^TG\Theta^{-1}(\hat{\Theta}-\Theta)D^{-1}(\hat{\Theta}-\Theta)\Theta^{-1}G^T\sigma\\
    &\le\frac{16C_{\delta}}{\rho_{\min}(R)^2 t_{\delta,R}}\rho(R)^{2}\rho_{\min}(R)^{-1}(\hat{\sigma}^T\hat{G}-\sigma^TG)\Theta^{-1}(\hat{G}^T\hat{\sigma}-G^T\sigma)\\
    &\quad+\frac{16C_{\delta}}{\rho_{\min}(R)^2 t_{\delta,R}}\sigma^TG\Theta^{-1}(\hat{\Theta}-\Theta)D^{-1}(\hat{\Theta}-\Theta)\Theta^{-1}G^T\sigma.
\end{align*}
Finally, using the fact that $t_{\delta,R}\geq 16\frac{C_\delta}{\rho_{\min}(R)}$,
\begin{equation} \label{eq:T12}
  T_{12}\le\frac{\rho(R)^{2}}{\rho_{\min}(R)^2}\III+\frac{1}{\rho_{\min}(R)}\II,  
\end{equation}
where the terms $\II$ and $\III$ have been introduced respectively in \cref{lem:termII} and \cref{lem:termIII}.

\paragraph{Control of the term $T_{11}$.}

First, we have
\begin{align*}
    T_{11}&={2\,\hat{\sigma}^T\hat{G}\Theta^{-1}(\hat{G}-G)^T\Sigma(\hat{G}-G)\Theta^{-1}\hat{G}^T\hat{\sigma}}\\
    &\le 4\,\sigma^TG\Theta^{-1}(\hat{G}-G)^T\Sigma(\hat{G}-G)\Theta^{-1}G^T\sigma\\
    &\quad+4(\hat{\sigma}^T\hat{G}-\sigma^TG)\Theta^{-1}(\hat{G}-G)^T\Sigma(\hat{G}-G)\Theta^{-1}(\hat{G}^T\hat{\sigma}-G^T\sigma)\\
    &\le 4\,\I+4\,(\hat{\sigma}^T\hat{G}-\sigma^TG)D^{-\frac{1}{2}}R^{-1}\bigg(D^{-\frac{1}{2}}(\hat{G}-G)^T\Sigma(\hat{G}-G)D^{-\frac{1}{2}}\bigg)R^{-1}D^{-\frac{1}{2}}(\hat{G}^T\hat{\sigma}-G^T\sigma)\\
    & \leq 4\,\I+4\,\rho\bigg(D^{-\frac{1}{2}}(\hat{G}-G)^T\Sigma(\hat{G}-G)D^{-\frac{1}{2}}\bigg)\times (\hat{\sigma}^T\hat{G}-\sigma^TG)D^{-\frac{1}{2}}R^{-1}R^{-1}D^{-\frac{1}{2}}(\hat{G}^T\hat{\sigma}-G^T\sigma).
\end{align*}

Using successively \cref{Cor:Majoration3Termes}, \cref{eq:inequA} and \ref{ass:signal},
\begin{align*}
    T_{11}
    &\le 4\,\I+4 \,\frac{C_{\delta}}{t_{\delta,R}}(\hat{\sigma}^T\hat{G}-\sigma^TG)D^{-\frac{1}{2}}R^{-1}R^{-1}D^{-\frac{1}{2}}(\hat{G}^T\hat{\sigma}-G^T\sigma)\\
    &\le 4\I
    + 4\frac{C_{\delta}}{t_{\delta,R}}\rho_{\min}(R)^{-1}(\hat{\sigma}^T\hat{G}-\sigma^TG)\Theta^{-1}(\hat{G}^T\hat{\sigma}-G^T\sigma).
\end{align*}
That is,
\begin{equation*}
    T_{11}\le 4\,\I+4\frac{C_{\delta}}{t_{\delta,R}\rho_{\min}(R)}\III.
\end{equation*}
Since $t_{\delta,R}\geq 16\frac{C_\delta}{\rho_{\min}(R)}$, it follows that 
\begin{equation} \label{eq:T11}
   T_{11}\le 4\,\I+\frac{1}{4}\III.
\end{equation}

\paragraph{Final bound on $T_1$.}

We deduce from \eqref{eq:T11} and \eqref{eq:T12} that
\begin{equation} \label{ineq:T1}
T_1\le 4\,\I+\frac{1}{\rho_{\min}(R)}\II+\big(\frac{\rho(R)^2}{\rho_{\min}(R)^2}+\frac{1}{4}\big)\III.
\end{equation}

\subsection[Bound on the second term]{Bound on $T_2$}
Using \ref{ass:inverse} and \cref{lem:R}, we obtain
\begin{align*}
    T_2&={\hat{\sigma}^T\hat{G}(\hat{\Theta}^{-1}-\Theta^{-1})\Theta(\hat{\Theta}^{-1}-\Theta^{-1})\hat{G}^T\hat{\sigma}}\\
    &=\hat{\sigma}^T\hat{G}\Theta^{-1}(\hat{\Theta}-\Theta)\hat{\Theta}^{-1}\Theta\hat{\Theta}^{-1}(\hat{\Theta}-\Theta)\Theta^{-1}\hat{G}^T\hat{\sigma}\\
    &\le\rho(R)\frac{4}{\rho_{\min}(R)^2}\hat{\sigma}^T\hat{G}\Theta^{-1}(\hat{\Theta}-\Theta)D^{-1}(\hat{\Theta}-\Theta)\Theta^{-1}\hat{G}^T\hat{\sigma}\\
    &\le2\rho(R)\frac{4}{\rho_{\min}(R)^2}\sigma^TG\Theta^{-1}(\hat{\Theta}-\Theta)D^{-1}(\hat{\Theta}-\Theta)\Theta^{-1}G^T\sigma\\
    &\quad+2\rho(R)\frac{4}{\rho_{\min}(R)^2}(\hat{\sigma}^T\hat{G}-\sigma^TG)\Theta^{-1}(\hat{\Theta}-\Theta)D^{-1}(\hat{\Theta}-\Theta)\Theta^{-1}(\hat{G}^T\hat{\sigma}-G^T\sigma)\\
    &\le2\rho(R)\frac{4}{\rho_{\min}(R)^2}\II %\\ &\quad
    +2\rho(R)\frac{4}{\rho_{\min}(R)^2}(\hat{\sigma}^T\hat{G}-\sigma^TG)D^{-\frac{1}{2}}R^{-1}(\hat{R}-R)^2R^{-1}D^{-\frac{1}{2}}(\hat{G}^T\hat{\sigma}-G^T\sigma).
\end{align*}
Using again \cref{lem:R},
\begin{align}
    T_2    &\le2\rho(R)\frac{4}{\rho_{\min}(R)^2}\II+2\rho(R)\frac{4}{\rho_{\min}(R)^2}\rho(R)^2\rho_{\min}(R)^{-1}(\hat{\sigma}^T\hat{G}-\sigma^TG)\Theta^{-1}(\hat{G}^T\hat{\sigma}-G^T\sigma) \nonumber\\
    &\le 8\frac{\rho(R)}{\rho_{\min}(R)^2}\II+8\frac{\rho(R)^3}{\rho_{\min}(R)^{3}}\III.
    \label{ineq:T2}
\end{align}

\subsection{End of the proof}
Using \eqref{ineqPrincipale}, \eqref{ineq:T1} and \eqref{ineq:T2}, we obtain that
\begin{align*}
\frac{1}{n} \| X\hat\beta_K - X\bar{\beta}\|^2
& = (\hat\beta_K - \bar{\beta})^T \Sigma (\hat\beta_K - \bar{\beta})\\
& \leq 16\I+\frac{4}{\rho_{\min}(R)}\big(1+8\Cond(R)\big)\II + \big(3+4\Cond(R)^2+32\Cond(R)^3 \big) \III.
\end{align*}

Using \cref{lem:termI}, \cref{lem:termII}, \cref{lem:termIII}, we get, on the event $\mathcal{A}_\delta$, 
\begin{align}
\lefteqn{\frac{1}{n} \| X\hat\beta_K - X\bar\beta\|^2} \nonumber \\
 &   \leq C_\delta\left(16+ \frac{1}{2} +  4\,\Cond(R)\right)  \frac{\tau^2}{n} \left(\sum_{i=1}^{K}\sqrt{\Tr(\Sigma^{2i})}\lvert \Lambda_i\rvert\right)^2 \nonumber \\
 & \quad + 8C_\delta^2\,\Cond(R) \left( 1 + 8\,\Cond(R) \right) \frac{\tau^2}{n}\|\tilde{\Lambda}\|^2 \sum_{j=1}^{K}\frac{\rho(\Sigma)^{2j}}{\sigma^T\Sigma^{2j-1}\sigma} \nonumber \\
 &  \quad + 2\frac{C_\delta^2}{\rho_{\min}(R)} \left( 3 + 4\,\Cond(R)^2+32\,\Cond(R)^3\right) \frac{\tau^2}{n} \left(\frac{1}{K t_{\delta,R}}\sum_{i=1}^{K}\bar\Lambda_i\Tr(\Sigma^{i})+\sum_{i=1}^{K}\bar{\Lambda}_{i}\rho(\Sigma)^{i}\right).
\label{eq:bornebonus}
\end{align}

Now, remark that 
\begin{equation}\label{Eq:SbarL}
\frac{1}{K t_{\delta,R}}\sum_{i=1}^{K}\bar\Lambda_i\Tr(\Sigma^{i})+\sum_{i=1}^{K}\bar{\Lambda}_{i}\rho(\Sigma)^{i} \leq \left( \frac{1}{t_{\delta,R}}+1 \right) \| \bar{\Lambda}\| \left( \sum_{i=1}^K( \Tr (\Sigma^i))^2 \right)^{1/2}.
\end{equation}
Moreover
\begin{align}
\| \bar \Lambda\|^2    
 & = \| D^{-1} G^T \sigma \|^2,\nonumber \\
 & = \sigma^T G D^{-2} G^T \sigma, \nonumber \\
 & = \Lambda^T \Theta D^{-2} \Theta \Lambda, \nonumber\\
 & = \Lambda^T D^{1/2} R D^{-1} R D^{\frac{1}{2}}\Lambda, \nonumber\\
 & \leq \Cond(D) \rho(R)^2\lVert {\Lambda}\rVert^{2}.\label{Eq:BarL}
\end{align}
In the same time, recalling that $\tilde \Lambda = D^{1/2} \Lambda$, we get  
\begin{equation}\label{Eq:TildeL}
\|\tilde{\Lambda}\|^2\sum_{j=1}^{K}\frac{\rho(\Sigma)^{2j}}{\sigma^T\Sigma^{2j-1}\sigma}\le\frac{\underset{i}{\max}(\sigma^T\Sigma^{2i-1}\sigma)}{\underset{j}{\min}(\sigma^T\Sigma^{2j-1}\sigma)}~\|\Lambda\|^2\sum_{j=1}^{K}\rho(\Sigma)^{2j}.\end{equation}

Plugging inequalities \eqref{Eq:SbarL}, \eqref{Eq:BarL} and \eqref{Eq:TildeL}  in bound \eqref{eq:bornebonus} leads to 
\begin{align*}
%\label{eq:borne_finale}
\lefteqn{\frac{1}{n} \| X\hat\beta_K - X\bar \beta\|^2} \\
 & \leq  C_\delta\left(17 +  4\,\Cond(R)\right)  \frac{\tau^2}{n} \left(\sum_{i=1}^{K}\sqrt{\Tr(\Sigma^{2i})}\lvert \Lambda_i\rvert\right)^2 \nonumber\\
 &  \; + 8C_\delta^2\,\Cond(R) \left( 1 + 8\,\Cond(R) \right) \Cond(D)\|\Lambda\|^2 \frac{\tau^2}{n}\sum_{j=1}^{K}\rho(\Sigma)^{2j}\\
 & \; + 2{C_\delta}\Bigl(C_\delta{\Cond(R)}+\frac{\rho(R)}{16}\Bigr) \left( 3 + 4\Cond(R)^2+32\Cond(R)^3\right) \frac{\tau^2}{n} \sqrt{\Cond(D)} \| \Lambda\| \Bigl( \sum_{i=1}^K( \Tr (\Sigma^i))^2 \Bigr)^{\frac{1}{2}}\\
& \le C_\delta(21+72C_\delta)\Cond(R)^2\Cond(D) \frac{\tau^2}{n} \|\Lambda\|^2 \sum_{i=1}^{K}\Tr(\Sigma^{2i})\\
& \quad +78\,C_\delta(C_\delta+\rho(R)/16)\Cond(R)^4\sqrt{\Cond(D)}\| \Lambda\| \frac{\tau^2}{n} \left( \sum_{i=1}^K( \Tr (\Sigma^i))^2 \right)^{\frac{1}{2}}\\
 & =: D^{(1)}_{\delta,R} \frac{\tau^2}{n} \Cond(D)\|\Lambda\|^2 \sum_{i=1}^{K}\Tr(\Sigma^{2i})+ D^{(2)}_{\delta,R} \frac{\tau^2}{n} \sqrt{\Cond(D)}\| \Lambda\| \left( \sum_{i=1}^K( \Tr (\Sigma^i))^2 \right)^{\frac{1}{2}},
\end{align*}
with 
$$ D^{(1)}_{\delta,R} = C_\delta(21+72C_\delta)\Cond(R)^2 \text{ and }
 D^{(2)}_{\delta,R} = 78\,C_\delta(C_\delta+\rho(R)/16)\Cond(R)^4.$$
We highlight the term $\Cond(R)^{4}$ in the constant $D_{\delta,R}^{(2)}$.
The proof can be concluded according to the last bound and \eqref{eq:beginproof}, with 
\begin{equation}
D_{\delta,R} =  \max( D^{(1)}_{\delta,R},D^{(2)}_{\delta,R}   ).
\label{eq:constantD}
\end{equation}

\subsection{A more precise result}
\label{s:Bonus}

 The bound displayed in \cref{Th:1} has been simplified for the ease of exposition. However, a more precise bound can be extracted from the proof. The corresponding result is displayed below. 
 
\begin{theorem}\label[thm]{Th:3} 
Let $\delta \in (0,1)$. Suppose that \ref{ass:inverse} and \ref{ass:signal} hold. Then, with a probability larger than $1-\delta$, 
\begin{align*}
\lefteqn{\frac{1}{n}\|X\hat{\beta}_{K}-X\beta\|^2}\\
& \le\frac{2}{n}\underset{v\in[G]}{\inf}\|X(\beta-v)\|^2
    + \mathcal{D}^{(1)}_{\delta,R} \frac{\tau^2}{n} \left(\sum_{i=1}^{K}\sqrt{\Tr(\Sigma^{2i})}\lvert \Lambda_i\rvert\right)^2 \\
 &  \quad +  \mathcal{D}^{(2)}_{\delta,R}\|\tilde{\Lambda}\|^2 \frac{\tau^2}{n}\sum_{j=1}^{K}\frac{\rho(\Sigma)^{2j}}{\sigma^T\Sigma^{2j-1}\sigma} \\
 &  \quad + \mathcal{D}^{(3)}_{\delta,R} \frac{\tau^2}{n} \left(\frac{1}{K t_{\delta,R}}\sum_{i=1}^{K}\bar\Lambda_i\Tr(\Sigma^{i})+\sum_{i=1}^{K}\bar{\Lambda}_{i}\rho(\Sigma)^{i}\right).
 \end{align*}
 for some constants $\mathcal{D}^{(j)}_{\delta,R}$, $j\in \lbrace 1,2,3 \rbrace$ depending only from $\delta$ and $R$.
 \end{theorem}
\begin{proof}
We use \eqref{eq:beginproof} together with \eqref{eq:bornebonus} to immediately obtain the result with 
$$ \mathcal{D}^{(1)}_{\delta,R} =  32 C_\delta + 8\rho_{\min}(R)C_{\delta} \left( \frac{1}{8\,\rho_{\min}(R)} +  \frac{\rho(R)}{\rho_{\min}(R)^2} \right) ,$$
$$ \mathcal{D}^{(2)}_{\delta,R} = 128 C_\delta^2\rho(R) \left( \frac{1}{8\,\rho_{\min}(R)} + \frac{\rho(R)}{\rho_{\min}(R)^2} \right),$$
and
$$ \mathcal{D}^{(3)}_{\delta,R}= 4\frac{C_\delta^2}{\rho_{\min}(R)} \left( 2 + 32\frac{\rho(R)^3}{\rho_{\min}(R)^{3}}+\big(\frac{4\rho(R)^2}{\rho_{\min}(R)^2}+1\big)\right).$$
\end{proof}

%%%%%%%%%%%%%%%%%%%%%%%%%%%%%%%%%%%%%%%%%%%%%%%%%%%%%%%%%%%%%%
%%%%%%%%%%%%%%%%%%%%%%%%%%%%%%%%%%%%%%%%%%%%%%%%%%%%%%%%%%%%%%
%%%%%%%%%%%%%%%%%%%%%%%%%%%%%%%%%%%%%%%%%%%%%%%%%%%%%%%%%%%%%%

\section{Ridge regularization}
\label{s:proof_ridge}

\subsection{Notations and assumptions}
We keep the same notations as before. We recall that $\Delta_{\alpha}=\mathrm{diag(\alpha_i)}$ where $\alpha_{i}>0$ for all $i\in\{1,...,K\}$. Let $\Theta_{\alpha}=\Theta+\Delta_{\alpha}$ and respectively  $D_{\alpha}=\mathrm{diag}(\Theta_{\alpha})$, $R_{\alpha}=D_{\alpha}^{-\frac{1}{2}}\Theta_{\alpha}D_{\alpha}^{-\frac{1}{2}}$. We also introduce $\hat{\Theta}_{\alpha}=\hat{\Theta}+\Delta_{\alpha}$ and $\hat{R}_{\alpha}=D_{\alpha}^{-\frac{1}{2}}\hat{\Theta}_{\alpha}D_{\alpha}^{-\frac{1}{2}}$.
We define $\beta_{\alpha}=G\Theta_{\alpha}^{-1}G^T\sigma$ and $\hat{\beta}_{K,\alpha}=\hat{G}\hat{\Theta}^{-1}_{\alpha}\hat{G}^T\hat{\sigma}$. Let us denote $\Lambda_{\alpha}=\Theta_{\alpha}^{-1}G^T\sigma$, the theoretical coordinates of $\beta_{\alpha}$ relative to the Krylov space.

\subsection[Some properties on the Ridge correlation matrices]{Some properties of $R_{\alpha}$ and $\hat{R}_{\alpha}$.}

\subsubsection[]{Bounds on the spectrum of $R_{\alpha}$}
\begin{lemma}\label{Lem:RhoMinAlpha1}
Assume that \ref{ass:inverse} holds. Then, 
$$ \rho_{\min}(R_\alpha) > \rho_{\min}(R) \quad \forall \alpha \in (\mathbb{R}^+)^K.$$
\end{lemma}
\begin{proof}
First remark that 
\begin{align*}
R_\alpha 
& =  D_\alpha^{-1/2} \Theta_\alpha D_\alpha^{-1/2}\\
%& =  (D+\Delta_\alpha)^{-1/2} (\Theta+ \Delta_\alpha) (D+\Delta_\alpha)^{-1/2}\\
& =  (D+\Delta_\alpha)^{-1/2} \Theta (D+\Delta_\alpha)^{-1/2} +  (D+\Delta_\alpha)^{-1/2} \Delta_\alpha (D+\Delta_\alpha)^{-1/2}\\
& = (D+\Delta_\alpha)^{-1/2} D^{1/2} R D^{1/2} (D+\Delta_\alpha)^{-1/2} +  (D+\Delta_\alpha)^{-1/2} \Delta_\alpha (D+\Delta_\alpha)^{-1/2}.
\end{align*}
Then, for any $x\in \mathbb{R}^K$ such that $x^Tx=1$, we have
\begin{align*}
\lefteqn{x^T R_\alpha x }\\
& = x^T (D+\Delta_\alpha)^{-1/2} D^{1/2} R D^{1/2} (D+\Delta_\alpha)^{-1/2}x + x^T (D+\Delta_\alpha)^{-1/2} \Delta_\alpha (D+\Delta_\alpha)^{-1/2} x\\
& \geq \rho_{\min}(R) \, x^T (D+\Delta_\alpha)^{-1/2} D (D+\Delta_\alpha)^{-1/2}x + x^T (D+\Delta_\alpha)^{-1/2} \Delta_\alpha (D+\Delta_\alpha)^{-1/2} x\\
& \geq \min(1,\rho_{\min}(R)) \times \left( x^T (D+\Delta_\alpha)^{-1/2} D (D+\Delta_\alpha)^{-1/2}x + x^T (D+\Delta_\alpha)^{-1/2} \Delta_\alpha (D+\Delta_\alpha)^{-1/2} x\right)\\
& = \min(1,\rho_{\min}(R)),
\end{align*}
which proves the desired result. 
\end{proof}

The following lemma provides a more accurate bound.
\begin{lemma}\label{Lem:RhoMinAlpha}
Assume that \ref{ass:inverse} holds. Then, for any $\alpha \in (\mathbb{R}^+)^K$, for any $x\in\mathbb{R}^K$,
	$$\rho_{\min}(R_{\alpha})x^Tx\ge	\rho_{\min}(R)x^Tx+\left(1-\rho_{\min}(R)\right)x^T(D+\Delta_\alpha)^{-\frac{1}{2}}\Delta_{\alpha}(D+\Delta_\alpha)^{-\frac{1}{2}}x.$$
\end{lemma}
\begin{proof}
We have
\begin{align*}
	\lefteqn{x^T R_\alpha x } \nonumber\\
	& = x^T (D+\Delta_\alpha)^{-1/2} D^{1/2} R D^{1/2} (D+\Delta_\alpha)^{-1/2}x + x^T (D+\Delta_\alpha)^{-1/2} \Delta_\alpha (D+\Delta_\alpha)^{-1/2} x \nonumber \\
	& \geq \rho_{\min}(R) \times \left( x^T (D+\Delta_\alpha)^{-1/2} D (D+\Delta_\alpha)^{-1/2}x + x^T (D+\Delta_\alpha)^{-1/2} \Delta_\alpha (D+\Delta_\alpha)^{-1/2} x\right) \nonumber \\
	& \quad + (1-\rho_{\min}(R))x^T(D+\Delta_\alpha)^{-\frac{1}{2}}\Delta_{\alpha}(D+\Delta_\alpha)^{-\frac{1}{2}}x \\
	& \geq \rho_{\min}(R)x^Tx+\left(1-\rho_{\min}(R)\right)x^T(D+\Delta_\alpha)^{-\frac{1}{2}}\Delta_{\alpha}(D+\Delta_\alpha)^{-\frac{1}{2}}x.
\end{align*}
This proves \cref{Lem:RhoMinAlpha}.
\end{proof}
Actually, a straightforward consequence is that 	$$\rho_{\min}(R_{\alpha})\ge	\rho_{\min}(R)+\big(1-\rho_{\min}(R)\big)\underset{1\le i \le K}{\min}\left(\frac{\alpha_i}{\sigma^T\Sigma^{2i-1}\sigma+\alpha_i}\right).$$
Yet, the formulation in \cref{Lem:RhoMinAlpha} is more useful in the following.

We can also provide an upper bound on the spectral radius of $R_\alpha$.
\begin{lemma}\label{Lem:RhoMaxRalpha}
For any $\alpha \in (\mathbb{R}^+)^K$, we have
$$\rho(R_{\alpha})\le \rho(R).$$
\end{lemma}
\begin{proof}
Let $x\in \mathbb{R}^K$ such that $x^Tx=1$. Then,
\begin{align*}
    x^TR_{\alpha}x&=x^TD_{\alpha}^{-\frac{1}{2}}\Theta_{\alpha}D_{\alpha}^{-\frac{1}{2}}x\\
    &=x^TD_{\alpha}^{-\frac{1}{2}}D^{\frac{1}{2}}RD^{\frac{1}{2}}D_{\alpha}^{-\frac{1}{2}}x+x^TD_{\alpha}^{-\frac{1}{2}}\Delta_{\alpha}D_{\alpha}^{-\frac{1}{2}}x\\
    &\le\rho(R)\,x^TD_{\alpha}^{-\frac{1}{2}}DD_{\alpha}^{-\frac{1}{2}}x+x^TD_{\alpha}^{-\frac{1}{2}}\Delta_{\alpha}D_{\alpha}^{-\frac{1}{2}}x\\
    &\le \max(1,\rho(R))\,x^TD_{\alpha}^{\frac{1}{2}}(D+\Delta_{\alpha})D^{-\frac{1}{2}}_{\alpha}x\\
    &\le \max(1,\rho(R)).
\end{align*}
Note that $\Tr(R)=K\leq K\rho(R)$. Consequently, $\rho(R)\geq 1$.
\end{proof}

\subsubsection[]{Inversion of $\hat{R}_{\alpha}$}

We first state that the inversion of the matrix $\hat{R}_{\alpha}$ exists with high probability.

\begin{lemma}
\label{lem:invertR}
Assume that \ref{ass:inverse} holds and consider $\alpha$ in \eqref{eq:alphai} with $c_{\delta}\geq 16C_\delta$. Then, on the event $\mathcal{A}_\delta$ defined in  \cref{Prop:Majoration3Termes}, we have
\begin{align*}
    \rho_{\min}(\hat{R}_{\alpha}) \ge \frac{\rho_{\min}(R)}{2} \quad \mathrm{and} \quad   \rho(\hat{R}_{\alpha}-R_{\alpha})\le\rho(R).
\end{align*}
\end{lemma}
\begin{proof}
Let $x\in\RR^{K}$ such that $x^Tx=1$ and denote $y=D_{\alpha}^{-\frac{1}{2}}x$. Then
\begin{align*}
    x^T\hat{R}_{\alpha}x&=y^T\hat{G}^T\Sigma\hat{G}y+y^T\Delta_{\alpha}y\\
    &=y^T(\hat{G}-G)^T\Sigma(\hat{G}-G)y+y^T\big(G^T\Sigma G+\Delta_{\alpha}\big)y+2y^T(\hat{G}-G)^{T}\Sigma G y\\
	&\ge y^T\Theta_{\alpha}y-2|y^T(\hat{G}-G)^T\Sigma Gy|.
\end{align*}
Applying inequality $2ab\leq a^2+b^2$ with well chosen $a,b$, we get
\begin{align*}
	x^T\hat{R}_{\alpha}x&\ge y^T\Theta_{\alpha}y-\frac{1}{4}y^T\big(G^T\Sigma G\big)y-4y^T(\hat{G}-G)^T\Sigma(\hat{G}-G)y\\
	&\ge y^T\Theta y+y^T\Delta_{\alpha}y-\frac{1}{4} y^T\Theta y-4y^T(\hat{G}-G)^T\Sigma(\hat{G}-G)y\\
	&\ge \frac{3}{4}y^T\Theta_{\alpha}y-4y^T(\hat{G}-G)^T\Sigma(\hat{G}-G)y.
\end{align*}
Considering the fact that $y^T\Theta_{\alpha}y=x^TR_{\alpha}x$, we can use \cref{Lem:RhoMinAlpha} to have
$$y^T\Theta_{\alpha}y\ge \rho_{\min}(R)+(1-\rho_{\min}(R))y^T\Delta_{\alpha}y.$$
We obtain 
\begin{align*}
    x^T\hat{R}_{\alpha}x&\ge \frac{3}{4}\rho_{\min}(R)+\frac{3}{4}(1-\rho_{\min}(R))y^T\Delta_{\alpha}y-4y^T(\hat{G}-G)^T\Sigma(\hat{G}-G)y\\
    &\ge \frac{3}{4}\rho_{\min}(R)+\frac{3}{4}(1-\rho_{\min}(R))\left(y^T\Delta_{\alpha}y-\frac{16}{3}y^T(\hat{G}-G)^T\Sigma(\hat{G}-G)y\right)\\
    &\quad -4\rho_{\min}(R)y^T(\hat{G}-G)^T\Sigma(\hat{G}-G)y.
\end{align*}
We apply \cref{Cor:Majoration3Termes} and Cauchy-Schwarz inequality to get, on $\A_{\delta}$,  
\begin{align*}
    y^T(\hat{G}-G)^{T}\Sigma(\hat{G}-G)y&=\sum_{i,j}^{K}y_iy_j(\hat{\sigma}-\sigma)^T\Sigma^{i+j-1}(\hat{\sigma}-\sigma)\\
    &\le\sum_{i,j}^{K}|y_i| |y_j| C_{\delta}\frac{\tau^2}{n}\Tr(\Sigma^{i+j})\\
    &\le \sum_{i,j}^{K}|y_i| |y_j| C_{\delta}\frac{\tau^2}{n} \sqrt{\Tr(\Sigma^{2i})}\sqrt{\Tr(\Sigma^{2j})}\\
    &\le \bigg(\sum_{i=1}^{K}y_i\sqrt{C_{\delta}\frac{\tau^2}{n}\Tr(\Sigma^{2i})}\bigg)^{2}\\
    &\le C_{\delta}K\frac{\tau^2}{n}\sum_{i=1}^{K}y_i^{2}\Tr(\Sigma^{2i}).
\end{align*}
Then, we have 
\begin{multline*}
    x^T\hat{R}_{\alpha}x\ge \frac{3}{4}\rho_{\min}(R)+\frac{3}{4}(1-\rho_{\min}(R))\left(\sum_{i=1}^{K}y_{i}^{2}(\alpha_i-\frac{16}{3}C_{\delta}K\frac{\tau^2}{n}\Tr(\Sigma^{2i}))\right)\\
    -\rho_{\min}(R)\left(\sum_{i=1}^{K}y_i^{2}4C_{\delta}K\frac{\tau^2}{n}\Tr(\Sigma^{2i})\right).
\end{multline*}
The definition of $\alpha_i$ simplifies the inequality to 
\begin{align*}
    x^T\hat{R}_{\alpha}x&\ge\frac{3}{4}\rho_{\min}(R)-\rho_{\min}(R)\left(\sum_{i=1}^{K}y_i^{2}4C_{\delta}K\frac{\tau^2}{n}\Tr(\Sigma^{2i})\right)\\
    &\ge \frac{3}{4}\rho_{\min}(R)-\rho_{\min}(R)\left(\sum_{i=1}^{K}x_i^2\frac{4C_{\delta}K\frac{\tau^2}{n}\Tr(\Sigma^{2i})}{\sigma^T\Sigma^{2i-1}\sigma+\alpha_i}\right)\\
    &\ge \frac{3}{4}\rho_{\min}(R)-\rho_{\min}(R)\left(\sum_{i=1}^{K}x_i^2\frac{4C_{\delta}K\frac{\tau^2}{n}\Tr(\Sigma^{2i})}{\alpha_i}\right)\\
    &\ge \frac{\rho_{\min}(R)}{2},
\end{align*}
where the last step results from the definition of $\alpha_i$. This proves the first part of the lemma.

Let us now prove the second part. We have 
\begin{align*}
    x^T(R_{\alpha}-\hat{R}_{\alpha})x&\le x^TD^{-\frac{1}{2}}_{\alpha}(\hat{\Theta}_{\alpha}-\Theta_{\alpha})D^{-\frac{1}{2}}_{\alpha}x\\
	&\le x^TD^{-\frac{1}{2}}_{\alpha}(\hat{G}^T\Sigma\hat{G}-G^T\Sigma G)D^{-\frac{1}{2}}_{\alpha}x\\
	&\le x^TD^{-\frac{1}{2}}_{\alpha}\big((\hat{G}-G)^T\Sigma(\hat{G}-G)-2G^T\Sigma(\hat{G}-G)\big)D^{-\frac{1}{2}}_{\alpha}x\\
	&\le x^TD^{-\frac{1}{2}}_{\alpha}(\hat{G}-G)^T\Sigma(\hat{G}-G)D^{-\frac{1}{2}}_{\alpha}x-2x^TD^{-\frac{1}{2}}_{\alpha}G^T\Sigma(\hat{G}-G)D^{-\frac{1}{2}}_{\alpha}x\\
	&\le \rho(D^{-\frac{1}{2}}_{\alpha}(\hat{G}-G)^T\Sigma(\hat{G}-G)D^{-\frac{1}{2}}_{\alpha})+\frac{1}{2}x^TD^{-\frac{1}{2}}_{\alpha}\Theta_{\alpha}D^{-\frac{1}{2}}_{\alpha}x\\ 
	&\quad + 2x^TD^{-\frac{1}{2}}_{\alpha}(\hat{G}-G)^T\Sigma(\hat{G}-G)D^{-\frac{1}{2}}_{\alpha}x\\
	&\le 3\rho(D^{-\frac{1}{2}}_{\alpha}(\hat{G}-G)^T\Sigma(\hat{G}-G)D^{-\frac{1}{2}}_{\alpha})+\frac{1}{2}\rho(R_{\alpha})
\end{align*}
Using \eqref{eq:inequA} and \cref{Cor:Majoration3Termes}, we get 
\begin{align*}
    \rho(D^{-\frac{1}{2}}_{\alpha}(\hat{G}-G)^T\Sigma(\hat{G}-G)D^{-\frac{1}{2}}_{\alpha})&\le K \underset{1\le i,j \le K}{\max}\frac{(\hat{\sigma}-\sigma)^T\Sigma^{i+j-1}(\hat{\sigma}-\sigma)}{\sqrt{\sigma^T\Sigma^{2i-1}\sigma+\alpha_i}\sqrt{\sigma^T\Sigma^{2j-1}\sigma+\alpha_j}}\\
    &\le K \underset{1\le i,j\le K}{\max}\frac{C_{\delta}\Tr(\Sigma^{i+j})}{\sqrt{\sigma^T\Sigma^{2i-1}\sigma+\alpha_i}\sqrt{\sigma^T\Sigma^{2j-1}\sigma+\alpha_j}}\\
    &\le C_{\delta} \underset{1\le i,j\le K}{\max}\frac{\sqrt{K\frac{\tau^2}{n}\Tr(\Sigma^{2i})}}{\sqrt{\alpha_i}}\frac{\sqrt{K\frac{\tau^2}{n}\Tr(\Sigma^{2j})}}{\sqrt{\alpha_j}}\\
    &\le \frac{C_{\delta}}{c_{\delta}}.
\end{align*}
Hence,
$$
    x^T(R_{\alpha}-\hat{R}_{\alpha})x
    \le3\frac{C_{\delta}}{c_{\delta}}+\frac{1}{2}\rho(R_{\alpha})\le \rho(R_{\alpha}),
    $$
where we used the fact that $\rho(R_{\alpha})\ge 1\ge 6\frac{C_{\delta}}{c_{\delta}}.$ We conclude with \cref{Lem:RhoMaxRalpha}.
\end{proof}

\subsection{Preliminary and technical results}
In this part, we propose bounds on three major terms that appears in the proof of \cref{Th:2} (see Section \ref{s:Th:2} below). 

\subsubsection{First term}
\begin{lemma}\label{lem:TermIalpha}
On the event $\mathcal{A}_\delta$ defined in  \cref{Prop:Majoration3Termes}, we have 
\begin{align*}
\mathrm{I}_{\alpha}
& := \Lambda_{\alpha}^T(G-\hat{G})^T\Sigma(G-\hat{G})\Lambda_{\alpha}\\
& \le 2C_{\delta}\frac{\tau^2}{n}\big(\sum_{i=1}^{K}|\Lambda_i|\sqrt{\Tr(\Sigma^{2i})}\big)^2+2\rho_{\min}(R)^{-2}\frac{C_\delta}{c_{\delta}}\|D_{\alpha}^{-\frac{1}{2}}\Delta_{\alpha}\Lambda\|^2.
\end{align*}
\end{lemma}
\begin{proof}
First,
\begin{align*}
    \I_\alpha&=\sigma^TG\Theta^{-1}\Theta\Theta^{-1}_{\alpha}(G-\hat{G})^T\Sigma(G-\hat{G})\Theta^{-1}_{\alpha}\Theta\Theta^{-1}G^T\sigma\\
	&=\Lambda^T(\Theta_{\alpha}-\Delta_{\alpha})\Theta_{\alpha}^{-1}(G-\hat{G})^T\Sigma(G-\hat{G})\Theta_{\alpha}^{-1}(\Theta_{\alpha}-\Delta_{\alpha})\Lambda\\
	&\le2\I+2\Lambda^T\Delta_{\alpha}\Theta_{\alpha}^{-1}(G-\hat{G})^T\Sigma(G-\hat{G})\Theta_{\alpha}^{-1}\Delta_{\alpha}\Lambda,
	\end{align*}
where the term $\I$ is defined in \cref{lem:termI} above. Then,with \cref{Lem:RhoMinAlpha1},
\begin{align*}
    \I_\alpha	&\le2\I+2\Lambda^T\Delta_{\alpha}D_{\alpha}^{-\frac{1}{2}}R_{\alpha}^{-1}D_{\alpha}^{-\frac{1}{2}}(G-\hat{G})^T\Sigma(G-\hat{G})D_{\alpha}^{-\frac{1}{2}}
	R_{\alpha}^{-1}D_{\alpha}^{-\frac{1}{2}}\Delta_{\alpha}\Lambda\\
	&\le 2\I + 2\rho\big(D_{\alpha}^{-\frac{1}{2}}(G-\hat{G})\Sigma(G-\hat{G})D_{\alpha}^{-\frac{1}{2}}\big)\times \rho_{\min}(R)^{-2}\times \|D_{\alpha}^{-\frac{1}{2}}\Delta_{\alpha}\Lambda\|^2\\
	&\le 2\I+2K\underset{1\le i,j\le K}{\max}\left\{\frac{(\hat{\sigma}-\sigma)^T\Sigma^{i+j-1}(\hat{\sigma}-\sigma)}{\sqrt{\sigma^T\Sigma^{2i-1}\sigma+\alpha_i}\sqrt{\sigma^T\Sigma^{2j-1}\sigma+\alpha_j}}\right\}\rho_{\min}(R)^{-2}\|D_{\alpha}^{-\frac{1}{2}}\Delta_{\alpha}\Lambda\|^2,
\end{align*}
where the last step results from equation \eqref{eq:inequA}. We apply Cauchy-Schwarz inequality and \cref{Cor:Majoration3Termes} to get
\begin{align*}
	\I_{\alpha}&\le2\I+2\underset{1\le i,j\le K}{\max}\left\{\sqrt{C_{\delta}\frac{K\frac{\tau^2}{n}\Tr(\Sigma^{2i})}{\alpha_i}}\sqrt{C_{\delta}\frac{K\frac{\tau^2}{n}\Tr(\Sigma^{2j})}{\alpha_j}}\right\}\rho_{\min}(R)^{-2}\|D_{\alpha}^{-\frac{1}{2}}\Delta_{\alpha}\Lambda\|^2\\
	&\le 2\I+2\rho_{\min}(R)^{-2}\frac{C_\delta}{c_{\delta}}\|D_{\alpha}^{-\frac{1}{2}}\Delta_{\alpha}\Lambda\|^2.
\end{align*}
The definition of $\alpha$ in \cref{eq:alphai} justifies the last step. 
We conclude by the bound on the term $\I$ given by \cref{lem:termI}.
\end{proof}

\subsubsection{Second term}

\begin{lemma}\label{lem:TermIIalpha}
On the event $\A_{\delta} $ we have 
\begin{align*}
\tilde{\mathrm{II}}_{\alpha}
& := \sigma^TG\Theta_{\alpha}^{-1}(\Theta_{\alpha}-\hat{\Theta}_{\alpha})D_{\alpha}^{-1}(\Theta_{\alpha}-\hat{\Theta}_{\alpha})\Theta_{\alpha}^{-1}G^T\sigma\\
& \le 2\mathrm{II}_{\alpha}+2\frac{\rho(R)^2}{\rho_{\min}(R)^2}\|D_{\alpha}^{-\frac{1}{2}}\Delta_{\alpha}\Lambda\|^2,
\end{align*}
where 
$$ \mathrm{II}_{\alpha}:=\Lambda^T(\Theta-\hat{\Theta})D_\alpha^{-1}(\Theta-\hat{\Theta})\Lambda.$$
\end{lemma}

\begin{proof}
Note that $\hat{\Theta}_{\alpha}-\Theta_{\alpha}=\hat{\Theta}-\Theta$. Hence
\begin{align*}
    \tilde{\mathrm{II}}_{\alpha}&=\sigma^TG\Theta_{\alpha}^{-1}(\Theta_{\alpha}-\hat{\Theta}_{\alpha})D_{\alpha}^{-1}(\Theta_{\alpha}-\hat{\Theta}_{\alpha})\Theta_{\alpha}^{-1}G^T\sigma\\
	&=\sigma^TG\Theta_{\alpha}^{-1}(\Theta-\hat{\Theta})D_{\alpha}^{-1}(\Theta-\hat{\Theta})\Theta_{\alpha}^{-1}G^T\sigma\\
	&=\sigma^TG\Theta^{-1}\Theta\Theta_{\alpha}^{-1}(\Theta-\hat{\Theta})D_{\alpha}^{-1}(\Theta-\hat{\Theta})\Theta_{\alpha}^{-1}\Theta\Theta^{-1}G^T\sigma\\
	&=\Lambda(\Theta_{\alpha}-\Delta_{\alpha})\Theta_{\alpha}^{-1}(\Theta-\hat{\Theta})D_{\alpha}^{-1}(\Theta-\hat{\Theta})\Theta_{\alpha}^{-1}(\Theta_{\alpha}-\Delta_{\alpha})\Lambda.
\end{align*}
By developing this inequality we end up with 
\begin{align*}
\tilde{\mathrm{II}}_{\alpha}
&\le2\Lambda^T(\Theta-\hat{\Theta})D^{-1}_{\alpha}(\Theta-\hat{\Theta})\Lambda+2\Lambda^T\Delta_{\alpha}\Theta_{\alpha}^{-1}(\Theta_{\alpha}-\hat{\Theta}_{\alpha})D^{-1}_{\alpha}(\Theta_{\alpha}-\hat{\Theta}_{\alpha})\Theta_{\alpha}^{-1}\Delta_{\alpha}\Lambda\\
& \le 2\mathrm{II}_{\alpha}+2\Lambda^T\Delta_{\alpha}D_{\alpha}^{-\frac{1}{2}}R_{\alpha}^{-1}(R_{\alpha}-\hat{R}_{\alpha})^2R_{\alpha}^{-1}D_{\alpha}^{-\frac{1}{2}}\Delta_{\alpha}\Lambda\\
& \leq 2\mathrm{II}_{\alpha}+2(\rho(R_{\alpha}-\hat{R}_{\alpha}))^2\times  (\rho_{\min} (R))^{-2}  \times\Lambda^T\Delta_{\alpha}D_{\alpha}^{-1}\Delta_{\alpha}\Lambda\\
&\le2\mathrm{II}_{\alpha}+2\rho_{\min}(R)^{-2}\rho(R)^{2}\|D_{\alpha}^{-\frac{1}{2}}\Delta_{\alpha}\Lambda\|^2,
\end{align*}
where the last step results from \cref{lem:invertR}.
\end{proof}

\begin{lemma}\label{lem:TermIIRidge}
On the event $\mathcal{A}_\delta$, we have
\begin{align*}
\mathrm{II}_{\alpha}
& := \Lambda^T(\Theta-\hat{\Theta})D_\alpha^{-1}(\Theta-\hat{\Theta})\Lambda \\ 
& \le 2\frac{C_{\delta}}{c_{\delta}}\frac{\tau^2}{n}\left(\sum_{j=1}^{K}\sqrt{\Tr(\Sigma^{2j})}|\Lambda_j|\right)^2\\
&\quad+2\frac{\tau^2}{n} C_{\delta}^2\rho(R)\sum_{j=1}^{K}\frac{\rho(\Sigma)^{2j}}{\sigma^T\Sigma^{2j-1}\sigma+\alpha_j}\|\tilde{\Lambda}\|^2+2\frac{C_{\delta}^2}{c_{\delta}}\rho(R)\sum_{j=1}^{K}\alpha_j\Lambda_j^2,
\end{align*}
with $\tilde \Lambda=D^{1/2}\Lambda$ (introduced in \cref{lem:termII}).
\end{lemma}
\begin{proof}
On the event $\mathcal{A}_\delta$, using \cref{Cor:Majoration3Termes}, we get
\begin{align*}
    \mathrm{II}_{\alpha}
    &=\sum_{k=1}^{K}\frac{1}{\sigma^T\Sigma^{2k-1}\sigma+\alpha_{k}}\left(\sum_{j=1}^{K}(\hat{\Theta}_{kj}-\Theta_{kj})\Lambda_{j}\right)^2\\
    &\le\sum_{k=1}^{K}\frac{1}{\sigma^T\Sigma^{2k-1}\sigma+\alpha_k}\left(\sum_{j=1}^{K}\left(C_{\delta}\frac{\tau^2}{n}\Tr(\Sigma^{j+k})+C_{\delta}\sqrt{\frac{\tau^2}{n}}\rho(\Sigma)^{\frac{j+k}{2}}\sqrt{\sigma^T\Sigma^{j+k-1}\sigma}\right)|\Lambda_{j}|\right)^2\\
   &\le 2\,C_\delta^2\sum_{k=1}^{K}\Big(\sum_{j=1}^{K}\sqrt{\frac{\frac{\tau^2}{n}\Tr(\Sigma^{2k})}{\sigma^T\Sigma^{2k-1}\sigma+\alpha_k}}\sqrt{\frac{\tau^2}{n}\Tr(\Sigma^{2j})}|\Lambda_j|\Bigr)^2\\
   &+ 2\,C_{\delta}^2\sum_{k=1}^{K}\left(\sum_{j=1}^{K}\sqrt{\frac{\frac{\tau^2}{n}\rho(\Sigma)^{j+k}\sigma^T\Sigma^{j+k-1}\sigma}{\sigma^T\Sigma^{2k-1}\sigma+\alpha_k}}|\Lambda_{j}|\right)^2.
\end{align*}
Based on the definition of $\alpha$ in \eqref{eq:alphai}, we deduce that 
\begin{align*} 
\mathrm{II}_{\alpha}
    &\le 2\frac{C_\delta^2}{c_{\delta}}\left(\sum_{j=1}^{K}\sqrt{\frac{\tau^2}{n}\Tr(\Sigma^{2j})}|\Lambda_j|\right)^2 \\ 
    & \quad + 2{C_\delta^2}\sum_{k=1}^{K}\left(\sum_{j=1}^{K}\sqrt{\frac{\frac{\tau^2}{n}\rho(\Sigma)^{k}\sigma^T\Sigma^{j+k-1}\sigma}{\sigma^T\Sigma^{2k-1}\sigma+\alpha_k}}\sqrt{\frac{\frac{\tau^2}{n}\rho(\Sigma)^{j}}{\sigma^T\Sigma^{2j-1}\sigma+\alpha_j}}(D_\alpha^{1/2})_j|\Lambda_{j}|\right)^2.
\end{align*}
Then, applying Cauchy-Schwarz inequality on the second term, we obtain
\begin{align*}
\mathrm{II}_{\alpha}
    &\le 2\frac{C_\delta^2}{c_{\delta}}\left(\sum_{j=1}^{K}\sqrt{\frac{\tau^2}{n}\Tr(\Sigma^{2j})}|\Lambda_j|\right)^2\\ 
    &\quad +2C_{\delta}^2\sum_{k=1}^{K}\sum_{j=1}^{K} \frac{\tau^2}{n}\frac{\rho(\Sigma)^{k}\sigma^T\Sigma^{k+j-1}\sigma}{\sigma^T\Sigma^{2k-1}\sigma+\alpha_k}\frac{\rho(\Sigma)^{j}}{\sigma^T\Sigma^{2j-1}\sigma+\alpha_j}\|D_{\alpha}^{\frac{1}{2}}\Lambda\|^2.
\end{align*}
We consider in the following the vector $v_{\alpha}\in\RR^{K}$ defined as 
$$ (v_{\alpha})_j=\frac{\rho(\Sigma)^{j}}{\sqrt{\sigma^T\Sigma^{2j-1}\sigma+\alpha_j}} \quad \forall j\in\lbrace 1,\dots, K \rbrace.$$ 
Note that $\|D_{\alpha}^{\frac{1}{2}}\Lambda\|^2= \Lambda^T(D + \Delta_\alpha)\Lambda = \|\tilde{\Lambda}\|^2+\|\Delta_{\alpha}^{\frac{1}{2}}\Lambda\|^2$. Hence,
\begin{align*}
\mathrm{II}_{\alpha}
    &\le 2\frac{C_\delta^2}{c_{\delta}}\left(\sum_{j=1}^{K}\sqrt{\frac{\tau^2}{n}\Tr(\Sigma^{2j})}|\Lambda_j|\right)^2
    +2C_{\delta}^2\frac{\tau^2}{n}v_{\alpha}^TR_{\alpha}v_{\alpha}\bigl(\|\tilde{\Lambda}\|^2+\|\Delta_{\alpha}^{\frac{1}{2}}\Lambda\|^2\bigr)\\
    &\le 2\frac{C_\delta^2}{c_{\delta}}\left(\sum_{j=1}^{K}\sqrt{\frac{\tau^2}{n}\Tr(\Sigma^{2j})}|\Lambda_j|\right)^2 + 2C_{\delta}^{2}\frac{\tau^2}{n}\|v_{\alpha}\|^2\rho(R)\bigl(\|\tilde{\Lambda}\|^2+\|\Delta_{\alpha}^{\frac{1}{2}}\Lambda\|^2\bigr),
\end{align*}    
where we have used \cref{Lem:RhoMaxRalpha}.
Then, using the definition of $v_{\alpha}$ we have
\begin{align*}
\mathrm{II}_{\alpha}
    &\le2\frac{C_\delta^2}{c_{\delta}}\left(\sum_{j=1}^{K}\sqrt{\frac{\tau^2}{n}\Tr(\Sigma^{2j})}|\Lambda_j|\right)^2+ 2C_{\delta}^{2}\frac{\tau^2}{n}\sum_{j=1}^{K}\frac{\rho(\Sigma)^{2j}}{\sigma^T\Sigma^{2j-1}\sigma+\alpha_j}\times \rho(R)\bigl(\|\tilde{\Lambda}\|^2+\|\Delta_{\alpha}^{\frac{1}{2}}\Lambda\|^2\bigr)\\
    &\le2\frac{C_\delta^2}{c_{\delta}}\left(\sum_{j=1}^{K}\sqrt{\frac{\tau^2}{n}\Tr(\Sigma^{2j})}|\Lambda_j|\right)^2\\
    &\quad+2C_{\delta}^{2}\frac{\tau^2}{n}\sum_{j=1}^{K}\frac{\rho(\Sigma)^{2j}}{\sigma^T\Sigma^{2j-1}\sigma+\alpha_j}\times \rho(R)\|\tilde{\Lambda}\|^2+2C_{\delta}^2\rho(R)\frac{\tau^2}{n} \sum_{j=1}^{K}\frac{\rho(\Sigma)^{2j}}{\alpha_j}\|\Delta_{\alpha}^{\frac{1}{2}}\Lambda\|^2.
\end{align*}
We conclude the proof using the definition of $\alpha$ in \eqref{eq:alphai} in the first term.
\end{proof}

\subsubsection{Third term}

\begin{lemma}\label{lem:TermIIIalpha}
On the event $\mathcal{A}_\delta$ defined in  \cref{Prop:Majoration3Termes}, we have  
\begin{align*}
\tilde{\mathrm{III}}_{\alpha}
& := (\hat{\sigma}^T\hat{G}-\sigma^TG)\Theta_{\alpha}^{-1}\Theta\Theta_{\alpha}^{-1}(\hat{G}^T\hat{\sigma}-G^T\sigma) \le \frac{\rho(R)}{\rho_{\min}(R)^2}\mathrm{III}_{\alpha},
\end{align*}
with
\begin{align*}
\mathrm{III}_{\alpha}
& := (\hat{\sigma}^T\hat{G}-\sigma^TG)^TD_{\alpha}^{-1}(\hat{G}^T\hat{\sigma}-G^T\sigma) \\
& \le 2C_\delta \frac{\tau^2}{n}\left(\frac{1}{K\, c_{\delta}}\sum_{j=1}^K \Tr(\Sigma^j)\bar{\Lambda}_j+\sum_{j=1}^{K}\rho(\Sigma)^{j}\bar{\Lambda_{j}}\right),
\end{align*}
and where $\bar{\Lambda}$ is defined in \eqref{eq:defLbar}.
\end{lemma}

\begin{proof}
We can write $\tilde\III_{\alpha}$ as 
\begin{align*}
\tilde\III_\alpha &= (\hat{G}^T\hat{\sigma}-G^T\sigma)^TD_\alpha^{-1/2}R_\alpha^{-1}D_\alpha^{-1/2}D^{1/2}RD^{1/2}D_\alpha^{-1/2}R_\alpha^{-1}D_\alpha^{-1/2}(\hat{G}^T\hat{\sigma}-G^T\sigma)\\
& = \left\lVert R^{1/2}D^{1/2}D_\alpha^{-1/2}R_\alpha^{-1}D_\alpha^{-1/2}(\hat{G}^T\hat{\sigma}-G^T\sigma)\right\rVert^2.
\end{align*}
Note that $DD_\alpha^{-1}$ is a diagonal matrix with entries in [0,1]. Hence, 
$$ \tilde\III_\alpha\leq  \frac{\rho(R)}{\rho_{\min}(R_\alpha)^2} (\hat{G}^T\hat{\sigma}-G^T\sigma)^TD_{\alpha}^{-1}(\hat{G}^T\hat{\sigma}-G^T\sigma).$$
The right hand side is equal to $\frac{\rho(R)}{\rho_{\min}(R_\alpha)^2}\III_\alpha$ which is bounded by $\frac{\rho(R)}{\rho_{\min}(R)^2}\III_\alpha$ by \cref{Lem:RhoMinAlpha1}.

Let us now study the term $\III_\alpha$.
Using \cref{Cor:Majoration3Termes}, on the set $\mathcal{A}_\delta$,
$$
\III_\alpha \le 2C_{\delta}\sum_{j=1}^{K}\frac{(\frac{\tau^2}{n})^2\Tr(\Sigma^j)^2+\frac{\tau^2}{n}\rho(\Sigma)^{j}\sigma^T\Sigma^{j-1}\sigma}{\sigma^T\Sigma^{2j-1}\sigma+\alpha_j}.$$
Then, using the expression of $\alpha$ in \eqref{eq:alphai} for the first term, and the definition of $\bar\Lambda$,
\begin{align*}
\III_\alpha 	&\le 2C_{\delta}\frac{\tau^2}{n}\sum_{j=1}^{K}\frac{\Tr(\Sigma^{j})}{K}\frac{K\frac{\tau^2}{n}\Tr(\Sigma^j)}{\alpha_j}+2C_{\delta}\frac{\tau^2}{n}\sum_{j=1}^{K}\rho(\Sigma)^{j}\bar{\Lambda}_{j}\\
	&\le 2\frac{C_{\delta}}{K\,c_{\delta}}\frac{\tau^2}{n}\sum_{j=1}^{K}\frac{\Tr(\Sigma^{j})}{\rho(\Sigma)^j}+2C_{\delta}\frac{\tau^2}{n}\sum_{j=1}^{K}\rho(\Sigma)^{j}\bar{\Lambda}_{j}\\
	&\le 2\frac{C_{\delta}}{K\,c_{\delta}}\frac{\tau^2}{n}\sum_{j=1}^{K}\bar{\Lambda}_{j}\Tr(\Sigma^{j})+2C_{\delta}\frac{\tau^2}{n}\sum_{j=1}^{K}\rho(\Sigma)^{j}\bar{\Lambda}_j,
\end{align*}
where last inequality results from \eqref{eq:ratioS}.
\cref{lem:TermIIIalpha} follows.
\end{proof}

\subsection[Proof of Theorem 4.1]{Proof of \cref{Th:2}}
\label{s:Th:2}

The introduction of a regularization matrix $\Delta_{\alpha}$ in the expression of $\hat\beta_{K,\alpha}$ (see \eqref{eq:betaridge}) induces a new bias. Indeed, introducing the parameter
\begin{equation}
\beta_\alpha = G\Theta_\alpha^{-1} G^T \sigma \quad \mathrm{with} \quad \Theta_\alpha = \Theta + \Delta_\alpha,
\label{eq:betaalpha}
\end{equation}
we obtain
$$ \beta - \hat\beta_{K,\alpha} = \beta - \bar \beta + \bar\beta - \beta_\alpha + \beta_\alpha - \hat\beta_{K,\alpha},$$
where $\bar\beta=G\Theta^{-1} G^T \sigma$ has been introduced in \eqref{Eq:barbeta}. This equality leads to 
\begin{align}
\frac{1}{n}\|X(\beta-\hat{\beta}_{K,\alpha})\|^2
&\le\frac{2}{n}\|X(\beta-\bar{\beta})\|^2+\frac{4}{n}\|X(\bar{\beta}-\beta_{\alpha})\|^2+\frac{4}{n}\|X(\beta_{\alpha}-\hat{\beta}_{K,\alpha})\|^2 \nonumber\\
& \leq \frac{2}{n}\inf_{v\in [G]}\|X(\beta-v)\|^2+\frac{4}{n}\|X(\bar{\beta}-\beta_{\alpha})\|^2+\frac{4}{n}\|X(\beta_{\alpha}-\hat{\beta}_{K,\alpha})\|^2.
\label{Eq:BiaisDecomp}
\end{align}
The first member of this inequality illustrates the distance - in terms of prediction - between the target $\beta$ and the Krylov space $[G]$: it exactly corresponds to the first term displayed in our bound. Te second represents the bias created by the addition of the regularization  term $\Delta_{\alpha}$, while the last one is related to the variance of the estimator.  Finally, we will focus on the last term to obtain an upper bound on the prediction.

Concerning the second term in the r.h.s. of \eqref{Eq:BiaisDecomp}, we have
\begin{align*}
		\frac{4}{n}\|X(\bar{\beta}-\beta_{\alpha})\|^2&\le4\sigma^TG(\Theta^{-1}-\Theta_{\alpha}^{-1})\Theta(\Theta^{-1}-\Theta_{\alpha}^{-1})G^T\sigma\\
		&\le 4 \sigma^TG\Theta^{-1}(\Theta-\Theta_{\alpha})\Theta_{\alpha}^{-1}\Theta\Theta_{\alpha}^{-1}(\Theta-\Theta_{\alpha})\Theta^{-1}G^T\sigma\\
		&\le 4\Lambda^T\Delta_{\alpha}\Theta_{\alpha}^{-1}(\Theta_{\alpha}-\Delta_{\alpha})\Theta_{\alpha}^{-1}\Delta_{\alpha}\Lambda\\
		&\le 4\Lambda^T\Delta_{\alpha}\Theta_{\alpha}^{-1}\Delta_{\alpha}\Lambda+4\Lambda^T\Delta_{\alpha}D_{\alpha}^{-\frac{1}{2}}R_{\alpha}^{-1}D_{\alpha}^{-\frac{1}{2}}\Delta_{\alpha}D_{\alpha}^{-\frac{1}{2}}R_{\alpha}^{-1}D_{\alpha}^{-\frac{1}{2}}\Delta_{\alpha}\Lambda\\
		&\le 4\rho_{\min}(R_{\alpha})^{-1}\|D_{\alpha}^{-\frac{1}{2}}\Delta_{\alpha}\Lambda\|^2+4\rho_{\min}(R_{\alpha})^{-2}\|D_{\alpha}^{-\frac{1}{2}}\Delta_{\alpha}\Lambda\|^2 \times \rho(D_\alpha^{-1/2} \Delta_\alpha D_{\alpha}^{-1/2}) \\
		& \leq 4(\rho_{\min}(R_{\alpha})^{-1} + \rho_{\min}(R_{\alpha})^{-2}) \|D_{\alpha}^{-\frac{1}{2}}\Delta_{\alpha}\Lambda\|^2.
\end{align*}
The last inequality comes from the fact that $D_\alpha^{-1/2} \Delta_\alpha D_{\alpha}^{-1/2}$ is a diagonal matrix with entries in [0,1], since $D_\alpha=D+\Delta_\alpha$. 
Then, remark that 
$$
\|D_{\alpha}^{-\frac{1}{2}}\Delta_{\alpha}\Lambda\|^2=\sum_{j=1}^{K}\frac{\alpha_j^2}{\sigma^T\Sigma^{2j-1}\sigma+\alpha_j}\Lambda_j^2\le\sum_{j=1}^{K}\alpha_j\Lambda_j^2.
$$
Using this result and \cref{Lem:RhoMinAlpha1}, we hence obtain 
\begin{align}
  \frac{4}{n}\|X(\bar{\beta}-\beta_{\alpha})\|^2 & \leq 4(\rho_{\min}(R_{\alpha})^{-1} + \rho_{\min}(R_{\alpha})^{-2}) \times \sum_{j=1}^{K}\alpha_j\Lambda_j^2\nonumber\\
  & \leq 4(\rho_{\min}(R)^{-1} + \rho_{\min}(R)^{-2}) \times \sum_{j=1}^{K}\alpha_j\Lambda_j^2.
\label{eq:BiasTerm}    
\end{align}

The remaining part of the proof is devoted to the control of the last term appearing in \eqref{Eq:BiaisDecomp}. We will use the following decomposition:
\begin{align*}
\hat\beta_{K,\alpha} - \beta_\alpha
& = 	\hat{G}\hat{\Theta}^{-1}_{\alpha}\hat{G}^T\hat{\sigma}-G\Theta^{-1}_{\alpha}G^T\sigma \\
&=(\hat{G}-G)\hat{\Theta}^{-1}_{\alpha}\hat{G}^T\hat{\sigma}+G(\hat{\Theta}^{-1}_{\alpha}\hat{G}^T\hat{\sigma}-\Theta^{-1}_{\alpha}G^T\sigma)\\
	&=(\hat{G}-G)\hat{\Theta}_{\alpha}^{-1}\hat{G}^T\hat{\sigma}+G(\hat{\Theta}^{-1}_{\alpha}-\Theta^{-1}_{\alpha})\hat{G}^T\hat{\sigma}+G\Theta^{-1}_{\alpha}(\hat{G}^T\hat{\sigma}-G^T\sigma).
\end{align*}
It yields
\begin{align*}
\lefteqn{\frac{1}{n} \| X(\hat\beta_{K,\alpha} - \beta_\alpha)\|^2 }  \\
    & = (\hat{\beta}_{\alpha}-\beta_{\alpha})^T\Sigma(\hat{\beta}_{\alpha}-\beta_{\alpha}) \\
    &\le 4\hat{\sigma}^T\hat{G}\hat{\Theta}^{-1}_{\alpha}(\hat{G}-G)^T\Sigma(\hat{G}-G)\hat{\Theta}^{-1}_{\alpha}\hat{G}^T\hat{\sigma} + 4 \hat{\sigma}^T\hat{G}(\hat{\Theta}^{-1}_{\alpha}-\Theta^{-1}_{\alpha})\Theta(\hat{\Theta}^{-1}_{\alpha}-\Theta^{-1}_{\alpha})\hat{G}^T\hat{\sigma}\\
    &\quad +2(\hat{\sigma}^T\hat{G}-\sigma^TG)\Theta_{\alpha}^{-1}\Theta\Theta^{-1}_{\alpha}(\hat{G}^T\hat{\sigma}-G^T\sigma)\\
    &\le4\underbrace{\hat{\sigma}^T\hat{G}\hat{\Theta}^{-1}_{\alpha}(\hat{G}-G)^T\Sigma(\hat{G}-G)\hat{\Theta}^{-1}_{\alpha}\hat{G}^T\hat{\sigma}}_{:=T_{1}^{\alpha}}+4\underbrace{\hat{\sigma}^T\hat{G}(\hat{\Theta}^{-1}_{\alpha}-\Theta^{-1}_{\alpha})\Theta(\hat{\Theta}^{-1}_{\alpha}-\Theta^{-1}_{\alpha})\hat{G}^T\hat{\sigma}}_{:=T_{2}^{\alpha}}+2\tilde{\III}_{\alpha},
\end{align*}
where the term $\tilde{\III}_{\alpha}$ is introduced in \cref{lem:TermIIIalpha}.
With \cref{lem:TermIIIalpha},
\begin{equation}\label{ineqPrincipalealpha}
    \frac{1}{n} \| X(\hat\beta_{K,\alpha} - \beta_\alpha)\|^2\leq T_1+T_2+2\frac{\rho(R)}{\rho_{\min}(R)^2}\III_\alpha.
\end{equation}

\subsubsection[Bound on the first term]{Bound on $T_1^{\alpha}$}
First consider the term $T_1^{\alpha}$ appearing in \eqref{ineqPrincipalealpha}. We decompose this term as follows: 
\begin{align}
   T_1^{\alpha}&=\hat{\sigma}^T\hat{G}\hat{\Theta}^{-1}_{\alpha}(\hat{G}-G)^T\Sigma(\hat{G}-G)\hat{\Theta}^{-1}_{\alpha}\hat{G}^T\hat{\sigma} \nonumber\\
    &\le 2\hat{\sigma}^T\hat{G}\Theta^{-1}_{\alpha}(\hat{G}-G)^T\Sigma(\hat{G}-G)\Theta^{-1}_{\alpha}\hat{G}^T\hat{\sigma} \nonumber \\
     & \quad +2\hat{\sigma}^T\hat{G}(\hat{\Theta}^{-1}_{\alpha}-\Theta^{-1}_{\alpha})(\hat{G}-G)\Sigma(\hat{G}-G)(\hat{\Theta}^{-1}_{\alpha}-\Theta^{-1}_{\alpha})\hat{G}^T\hat{\sigma}\nonumber \\
    &=:T_{11}^{\alpha}+T_{12}^{\alpha}.
    \label{eq:decompT1a}
\end{align}

First, we concentrate our attention on the term $T_{12}^{\alpha}$. We have
\begin{align*}
   \lefteqn{T_{12}^{\alpha}}\\
   &= 2\hat{\sigma}^T\hat{G}(\hat{\Theta}^{-1}_{\alpha}-\Theta^{-1}_{\alpha})(\hat{G}-G)^T\Sigma(\hat{G}-G)(\hat{\Theta}^{-1}_{\alpha}-\Theta^{-1}_{\alpha})\hat{G}^T\hat{\sigma}\\
    &= 2\hat{\sigma}^T\hat{G}\Theta^{-1}_{\alpha}(\hat{\Theta}_{\alpha}-\Theta_{\alpha})\hat{\Theta}^{-1}_{\alpha}(\hat{G}-G)\Sigma(\hat{G}-G)\hat{\Theta}^{-1}_{\alpha}(\hat{\Theta}_{\alpha}-\Theta_{\alpha})\Theta^{-1}_{\alpha}\hat{G}^T\hat{\sigma}\\
    &= 2\hat{\sigma}^T\hat{G}\Theta^{-1}_{\alpha}(\hat{\Theta}_{\alpha}-\Theta_{\alpha})D^{-\frac{1}{2}}_{\alpha}\hat{R}_{\alpha}^{-1}\big(D^{-\frac{1}{2}}_{\alpha}(\hat{G}-G)^T\Sigma(\hat{G}-G)D^{-\frac{1}{2}}_{\alpha}\big)\hat{R}_{\alpha}^{-1}D^{-\frac{1}{2}}_{\alpha}(\hat{\Theta}_{\alpha}-\Theta_{\alpha})\Theta^{-1}_{\alpha}\hat{G}^T\hat{\sigma}\\
    &\le \frac{8}{\rho_{\min}(R)^2}\rho\bigg(D^{-\frac{1}{2}}_{\alpha}(\hat{G}-G)^T\Sigma(\hat{G}-G)D^{-\frac{1}{2}}_{\alpha}\bigg)\times \hat{\sigma}^T\hat{G}\Theta^{-1}_{\alpha}(\hat{\Theta}_{\alpha}-\Theta_{\alpha})D^{-1}_{\alpha}(\hat{\Theta}_{\alpha}-\Theta_{\alpha})\Theta^{-1}_{\alpha}\hat{G}^T\hat{\sigma}.
\end{align*}
where we have used \cref{lem:invertR}. Using \cref{eq:inequA} and \cref{Cor:Majoration3Termes}, we can remark that on the event $\mathcal{A}_\delta$
\begin{align}
    \rho\bigg(D^{-\frac{1}{2}}_{\alpha}(\hat{G}-G)^T\Sigma(\hat{G}-G))D^{-\frac{1}{2}}_{\alpha}\bigg) 
    & \leq K \max_{i,j} \left\{\frac{(\hat\sigma-\sigma)^T\Sigma^{i+j-1}(\hat\sigma -\sigma)}{\sqrt{\sigma^T\Sigma^{2i-1}\sigma+\alpha_i}\sqrt{\sigma^T\Sigma^{2j-1}\sigma+\alpha_j}}\right\} \nonumber\\
    & \leq C_\delta K \max_{i,j} \left\{\frac{\frac{\tau^2}{n}\Tr(\Sigma^{i+j})}{\sqrt{\sigma^T\Sigma^{2i-1}\sigma+\alpha_i}\sqrt{\sigma^T\Sigma^{2j-1}\sigma+\alpha_j}}\right\} \nonumber\\
    & \leq\frac{C_{\delta}}{c_{\delta}}, 
    \label{eq:1controlRho}
\end{align}
where the final step results from the definition of $\alpha$ (see \eqref{eq:alphai}).
Hence, 
$$
T_{12}^{\alpha}
\le \frac{8}{\rho_{\min}(R)^2}\frac{C_{\delta}}{c_{\delta}}\ \hat{\sigma}^T\hat{G}\Theta^{-1}_{\alpha}(\hat{\Theta}_{\alpha}-\Theta_{\alpha})D^{-1}_{\alpha}(\hat{\Theta}_{\alpha}-\Theta_{\alpha})\Theta^{-1}_{\alpha}\hat{G}^T\hat{\sigma}.$$
Then,
\begin{align*}
T_{12}^{\alpha}
    &\le \frac{16}{\rho_{\min}(R)^2}\frac{C_{\delta}}{c_{\delta}}(\hat{\sigma}^T\hat{G}-\sigma^TG)\Theta^{-1}_{\alpha}(\hat{\Theta}_{\alpha}-\Theta_{\alpha})D^{-1}_{\alpha}(\hat{\Theta}_{\alpha}-\Theta_{\alpha})\Theta^{-1}_{\alpha}(\hat{G}^T\hat{\sigma}-G^T\sigma)\\
    &\quad+\frac{16}{\rho_{\min}(R)^2}\frac{C_{\delta}}{c_{\delta}}\sigma^TG\Theta^{-1}_{\alpha}(\hat{\Theta}_{\alpha}-\Theta_{\alpha})D^{-1}_{\alpha}(\hat{\Theta}_{\alpha}-\Theta_{\alpha})\Theta^{-1}_{\alpha}G^T\sigma\\
    &\leq \frac{16}{\rho_{\min}(R)^2}\frac{C_{\delta}}{c_{\delta}}(\hat{\sigma}^T\hat{G}-\sigma^TG)D^{-\frac{1}{2}}_{\alpha}R^{-1}_{\alpha}(\hat{R}_{\alpha}-R_{\alpha})^{2}R^{-1}_{\alpha}D^{-\frac{1}{2}}_{\alpha}(\hat{G}^T\hat{\sigma}-G^T\sigma)\\
    &\quad+\frac{16}{\rho_{\min}(R)^2}\frac{C_{\delta}}{c_{\delta}}\sigma^TG\Theta^{-1}_{\alpha}(\hat{\Theta}_{\alpha}-\Theta_{\alpha})D^{-1}_{\alpha}(\hat{\Theta}_{\alpha}-\Theta_{\alpha})\Theta^{-1}_{\alpha}G^T\sigma.
\end{align*}
Using \cref{lem:invertR}, we obtain
\begin{align}
    T_{12}^{\alpha} & \le 16\rho_{\min}(R)^{-4}\rho(R)^{2}\frac{C_{\delta}}{c_{\delta}}(\hat{\sigma}^T\hat{G}-\sigma^TG)D^{-1}_{\alpha}(\hat{G}^T\hat{\sigma}-G^T\sigma)\nonumber \\
    &\quad + \frac{16}{\rho_{\min}(R)^2}\frac{C_{\delta}}{c_{\delta}}\sigma^TG\Theta^{-1}_{\alpha}(\hat{\Theta}_{\alpha}-\Theta_{\alpha})D^{-1}_{\alpha}(\hat{\Theta}_{\alpha}-\Theta_{\alpha})\Theta^{-1}_{\alpha}G^T\sigma \nonumber\\
    &\leq  \frac{16\rho(R)^2}{\rho_{\min}(R)^{4}}\frac{C_\delta}{c_{\delta}}\III_{\alpha}+\frac{16}{\rho_{\min}(R)^2}\frac{C_\delta}{c_{\delta}}\tilde{\II}_{\alpha},
    \label{eq:contT12}
\end{align}
where the terms $\III_{\alpha}$ and $\tilde{\II}_{\alpha}$ have been respectively introduced in  \cref{lem:TermIIIalpha} and \cref{lem:TermIIalpha}.

Now, we propose a bound for the term $T_{11}^{\alpha}$. First,
\begin{align*}
    T_{11}^{\alpha} & = 2\hat{\sigma}^T\hat{G}\Theta^{-1}_{\alpha}(\hat{G}-G)^T\Sigma(\hat{G}-G)\Theta^{-1}_{\alpha}\hat{G}^T\hat{\sigma}\\
    &\le4\sigma^TG\Theta^{-1}_{\alpha}(\hat{G}-G)^T\Sigma(\hat{G}-G)\Theta^{-1}_{\alpha}G^T\sigma\\
    &\quad+4(\hat{\sigma}^T\hat{G}-\sigma^TG)\Theta^{-1}_{\alpha}(\hat{G}-G)^T\Sigma(\hat{G}-G)\Theta^{-1}_{\alpha}(\hat{G}^T\hat{\sigma}-G^T\sigma)\\
    & = 4\I_{\alpha}
    +4(\hat{\sigma}^T\hat{G}-\sigma^TG)D^{-\frac{1}{2}}_{\alpha}R^{-1}_{\alpha}\bigg(D^{-\frac{1}{2}}_{\alpha}(\hat{G}-G)^T\Sigma(\hat{G}-G)D^{-\frac{1}{2}}_{\alpha}\bigg)R^{-1}_{\alpha}D^{-\frac{1}{2}}_{\alpha}(\hat{G}^T\hat{\sigma}-G^T\sigma),
\end{align*}
where the term $\I_{\alpha}$ has been introduced in \cref{lem:TermIalpha}. Using \eqref{eq:1controlRho}, we obtain
\begin{align}
T_{11}^\alpha
    &\le4\I_{\alpha}+4(\hat{\sigma}^T\hat{G}-\sigma^TG)D^{-\frac{1}{2}}_{\alpha}R^{-1}_{\alpha}R^{-1}_{\alpha}D^{-\frac{1}{2}}_{\alpha}(\hat{G}^T\hat{\sigma}-G^T\sigma) \times \rho\bigg(D^{-\frac{1}{2}}_{\alpha}(\hat{G}-G)^T\Sigma(\hat{G}-G)D^{-\frac{1}{2}}_{\alpha}\bigg), \nonumber\\
    &\le4\I_{\alpha}
    +4\frac{C_{\delta}}{c_{\delta}}\rho_{\min}(R_{\alpha})^{-2}\times (\hat{\sigma}^T\hat{G}-\sigma^TG)D^{-1}_{\alpha}(\hat{G}^T\hat{\sigma}-G^T\sigma), \nonumber\\
    & \leq 4\I_{\alpha}+4\frac{C_{\delta}}{c_{\delta}}\rho_{\min}(R)^{-2}\mathrm{III}_{\alpha}.
    \label{eq:contT11}
\end{align}
where we have used \cref{Lem:RhoMinAlpha1}, and with the term $\mathrm{III}_{\alpha}$ has been introduced in \cref{lem:TermIIIalpha}.

Finally, we deduce from \eqref{eq:decompT1a}, \eqref{eq:contT12} and  \eqref{eq:contT11} that
\begin{equation} \label{ineq:T1alpha}
    T_1^{\alpha}\le 4\I_{\alpha}+\frac{16}{\rho_{\min}(R)^2}\frac{C_{\delta}}{c_{\delta}}\tilde{\II}_{\alpha}+4\frac{C_{\delta}}{c_{\delta}}\bigg(1+4\frac{\rho(R)^{2}}{\rho_{\min}(R)^{2}}\bigg)\rho_{\min}(R)^{-2}\mathrm{III}_{\alpha}.
\end{equation}

\subsubsection[Bound on the second term]{Bound on $T_{2}^{\alpha}$}
First,
\begin{align*}
    T_2^{\alpha}&={\hat{\sigma}^T\hat{G}(\hat{\Theta}^{-1}_{\alpha}-\Theta^{-1}_{\alpha})\Theta(\hat{\Theta}^{-1}_{\alpha}-\Theta^{-1}_{\alpha})\hat{G}^T\hat{\sigma}}\\
    &= \hat{\sigma}^T\hat{G}\Theta^{-1}_{\alpha}(\hat{\Theta}_{\alpha}-\Theta_{\alpha})\hat{\Theta}^{-1}_{\alpha}\Theta\hat{\Theta}^{-1}_{\alpha}(\hat{\Theta}_{\alpha}-\Theta_{\alpha})\Theta^{-1}_{\alpha}\hat{G}^T\hat{\sigma}\\
    &\le\rho(R)\hat{\sigma}^T\hat{G}\Theta_{\alpha}^{-1}(\hat{\Theta}_{\alpha}-\Theta_{\alpha})D_{\alpha}^{-\frac{1}{2}}\hat{R}_{\alpha}^{-1}D_{\alpha}^{-\frac{1}{2}}DD_{\alpha}^{-\frac{1}{2}}\hat{R}^{-1}_{\alpha}D_{\alpha}^{-\frac{1}{2}}(\hat{\Theta}_{\alpha}-\Theta)\Theta_{\alpha}^{-1}\hat{G}^T\hat{\sigma}\\
    &\le\rho(R)\rho(D_{\alpha}^{-\frac{1}{2}}DD_{\alpha}^{-\frac{1}{2}})\times \hat{\sigma}^T\hat{G}\Theta_{\alpha}^{-1}(\hat{\Theta}_{\alpha}-\Theta_{\alpha})D_{\alpha}^{-\frac{1}{2}}\hat{R}_{\alpha}^{-1}\hat{R}^{-1}_{\alpha}D_{\alpha}^{-\frac{1}{2}}(\hat{\Theta}_{\alpha}-\Theta)\Theta_{\alpha}^{-1}\hat{G}^T\hat{\sigma}\\
    &\le\rho(R)\times \hat{\sigma}^T\hat{G}\Theta^{-1}_{\alpha}(\hat{\Theta}_{\alpha}-\Theta_{\alpha})D^{-\frac{1}{2}}_{\alpha}\hat{R}_{\alpha}^{-2}D_{\alpha}^{-\frac{1}{2}}(\hat{\Theta}_{\alpha}-\Theta_{\alpha})\Theta^{-1}_{\alpha}\hat{G}^T\hat{\sigma}.
\end{align*}
Using \cref{lem:invertR}, we get
$$    T_2^{\alpha}\le\rho(R)\frac{4}{\rho_{\min}(R)^2}\times \hat{\sigma}^T\hat{G}\Theta^{-1}_{\alpha}(\hat{\Theta}_{\alpha}-\Theta_{\alpha})D^{-1}_{\alpha}(\hat{\Theta}_{\alpha}-\Theta_{\alpha})\Theta^{-1}_{\alpha}\hat{G}^T\hat{\sigma}.$$
Next, we introduce  the term $\tilde{\mathrm{II}}_{\alpha}$ defined in \cref{lem:TermIIRidge} as follows,
\begin{align*}
T_2^{\alpha}
    &\le\frac{8\rho(R)}{\rho_{\min}(R)^2}\tilde{\II}_{\alpha}+\frac{8\rho(R)}{\rho_{\min}(R)^2}(\hat{\sigma}^T\hat{G}-\sigma^TG)\Theta_{\alpha}^{-1}(\hat{\Theta}_{\alpha}-\Theta_{\alpha})D_{\alpha}^{-1}(\hat{\Theta}_{\alpha}-\Theta_{\alpha})\Theta_{\alpha}^{-1}(\hat{G}^T\hat{\sigma}-G^T\sigma)\\
    &\le \frac{8\rho(R)}{\rho_{\min}(R)^2}\tilde{\mathrm{II}}_{\alpha}+\frac{8\rho(R)}{\rho_{\min}(R)^2}(\hat{\sigma}^T\hat{G}-\sigma^TG)D_{\alpha}^{-\frac{1}{2}}R_{\alpha}^{-1}(\hat{R}_{\alpha}-R_{\alpha})^2R_{\alpha}^{-1}D_{\alpha}^{-\frac{1}{2}}(\hat{G}^T\hat{\sigma}-G^T\sigma).\\
\end{align*}
Using again \cref{lem:invertR},
 \begin{equation}
     T_2^{\alpha}\le \frac{8\rho(R)}{\rho_{\min}(R)^2}\tilde{\mathrm{II}}_{\alpha}+8\frac{\rho(R)^3}{\rho_{\min}(R)^{4}}\mathrm{III}_{\alpha},
    \label{ineq:T2alpha}
\end{equation}
where the term $\mathrm{III}_{\alpha}$ has been introduced in \cref{lem:TermIIIalpha}.

\subsubsection{End of the proof}
We consider in the following that we are on the event $\A_{\delta}$. Using \eqref{ineqPrincipalealpha}, \eqref{ineq:T1alpha} and \eqref{ineq:T2alpha} we obtain that
\begin{align*}
\frac{1}{n}\|X(\hat{\beta}_{K,\alpha}-\beta_{\alpha})\|^2&=(\hat{\beta}_{\alpha}-\beta_{\alpha})^T\Sigma(\hat{\beta}_{\alpha}-\beta_{\alpha})\\
&\le 16\,\I_{\alpha}+32\,\bigg(2\frac{C_{\delta}}{c_{\delta}}+{\rho(R)}\bigg)\rho_{\min}(R)^{-2}\tilde{\II}_{\alpha}\\
&\quad +\bigg(2{\rho(R)}+16\frac{C_{\delta}}{c_{\delta}}+64\frac{C_{\delta}}{c_{\delta}}\frac{\rho(R)^2}{\rho_{\min}(R)^2}+32\frac{\rho(R)^3}{\rho_{\min}(R)^2}\bigg)\rho_{\min}(R)^{-2}\III_{\alpha}.
\end{align*}
Using \cref{lem:TermIalpha}, \cref{lem:TermIIalpha} and \cref{lem:TermIIIalpha}, we get 
\begin{align}
    \lefteqn{\frac{1}{n}\|X(\hat{\beta}_{K,\alpha}-\beta_{\alpha})\|^2}\nonumber\\
    &\le 32\left(C_{\delta}+4\frac{C_\delta}{c_\delta}\big(2\frac{C_{\delta}}{c_{\delta}}+{\rho(R)}\big)\rho_{\min}(R)^{-2}\right)
    \frac{\tau^2}{n}\bigg(\sum_{i=1}^{K}\sqrt{\Tr(\Sigma^{2i})}|\Lambda_i|\bigg)^2\nonumber\\
    & \quad +\left(128\,C_{\delta}^2\rho(R)\bigg(2\frac{C_{\delta}}{c_{\delta}}+{\rho(R)}\bigg)\rho_{\min}(R)^{-2}\right)\|\tilde{\Lambda}\|^2\frac{\tau^2}{n}\sum_{j=1}^{K}\frac{\rho(\Sigma)^{2j}}{\sigma^T\Sigma^{2j-1}\sigma+\alpha_j}\nonumber\\
    & \quad +\left(16\frac{C_{\delta}}{c_{\delta}}+2{\rho(R)}+64\frac{C_{\delta}}{c_{\delta}}\frac{\rho(R)^2}{\rho_{\min}(R)^2}+32\frac{\rho(R)^3}{\rho_{\min}(R)^2}\right)\frac{2C_\delta}{\rho_{\min}(R_{\alpha})^{2}}\frac{\tau^2}{n}\sum_{j=1}^K \bar{\Lambda}_{j}\Big(\frac{\Tr(\Sigma^j)}{Kc_{\delta}}+\rho(\Sigma)^{j}\Big)\nonumber\\
    & \quad +\left(32\frac{C_{\delta}}{c_{\delta}}
    +64\frac{\rho(R)^{2}}{\rho_{\min}(R)^2}\big(2\frac{C_{\delta}}{c_{\delta}}+{\rho(R)}\big)\right)\rho_{\min}(R)^{-2}\|D_{\alpha}^{-\frac{1}{2}}\Delta_{\alpha}\Lambda\|^2\nonumber\\
    &  \quad + 128\frac{C_{\delta}^2}{c_{\delta}}\bigg(2\frac{C_{\delta}}{c_{\delta}}+{\rho(R)}\bigg)\frac{\rho(R)}{\rho_{\min}(R)^{2}}\sum_{j=1}^{K}\alpha_j\Lambda_j^2.\label{Eq:RidgeBound}
\end{align}
The two last line corresponds to the bias induced by the Ridge procedure.

Now, remark that
$$\frac{1}{K c_{\delta}}\sum_{j=1}^{K}\bar\Lambda_j\Tr(\Sigma^{j})+\sum_{j=1}^{K}\bar{\Lambda}_{j}\rho(\Sigma)^{j} \leq \left( \frac{1}{c_{\delta}}+1 \right) \| \bar{\Lambda}\| \left( \sum_{j=1}^K( \Tr (\Sigma^j))^2 \right)^{1/2}.$$
Moreover, with \eqref{Eq:BarL},
\begin{align*}
\| \bar \Lambda\|  \leq \Cond(D)^{1/2} \rho(R)\lVert {\Lambda}\rVert.
\end{align*}
In the same time, equation \eqref{Eq:TildeL} yields
$$\|\tilde{\Lambda}\|^2\sum_{j=1}^{K}\frac{\rho(\Sigma)^{2j}}{\sigma^T\Sigma^{2j-1}\sigma}\le\Cond(D)~\|\Lambda\|^2\sum_{j=1}^{K}\rho(\Sigma)^{2j}.$$

Plugging these three inequalities in the previous bound leads to 
\begin{align}  
{\frac{1}{n}\|X(\hat{\beta}_{K,\alpha}-\beta_{\alpha})\|^2}&\leq C_{\delta,R}^{(1)}\frac{\tau^2}{n}\Cond(D)\|\Lambda\|^2\sum_{i=1}^{K}\Tr(\Sigma^{2i}) \nonumber \\
&\quad + C_{\delta,R}^{(2)}\frac{\tau^2}{n}\Cond(D)^{\frac{1}{2}}\|\Lambda\|\bigg(\sum_{j=1}^{K}\Tr(\Sigma^{i})^{2}\bigg)^{\frac{1}{2}} + C_{\delta,R}^{(3)}\sum_{j=1}^{K}\alpha_j\Lambda_j^2,\label{Eq:Variance}
\end{align}
with 
$$C_{\delta,R}^{(1)}=32\left(C_{\delta}+4\frac{C_\delta}{c_\delta}\big(2\frac{C_{\delta}}{c_{\delta}}+{\rho(R)}\big)\rho_{\min}(R)^{-2}\right)+\left(128\,C_{\delta}^2\rho(R)\bigg(2\frac{C_{\delta}}{c_{\delta}}+{\rho(R)}\bigg)\rho_{\min}(R)^{-2}\right),$$
$$C_{\delta,R}^{(2)}=\left(16\frac{C_{\delta}}{c_{\delta}}+2{\rho(R)}+64\frac{C_{\delta}}{c_{\delta}}\frac{\rho(R)^2}{\rho_{\min}(R)^2}+32\frac{\rho(R)^3}{\rho_{\min}(R)^2}\right)\frac{2C_\delta}{\rho_{\min}(R)^{2}}\left( \frac{1}{c_{\delta}}+1 \right),$$%\frac{2C_{\delta}\rho(R)}{\rho_{\min}(R)^{2}},$$
and 
$$C_{\delta,R}^{(3)}=128\frac{C_{\delta}^2}{c_{\delta}}\bigg(2\frac{C_{\delta}}{c_{\delta}}+{\rho(R)}\bigg)\frac{\rho(R)}{\rho_{\min}(R)^{2}}.$$

We can highlight the term $\Cond(R)^{4}$ in the constants $C_{\delta,R}^{(1)},C_{\delta,R}^{(2)}$ and $C_{\delta,R}^{(3)}$.

Using \eqref{Eq:BiaisDecomp}, \eqref{eq:BiasTerm} and \eqref{Eq:Variance}, we finally obtain 
\begin{align*}
    \frac{1}{n}\|X(\hat{\beta}_{K,\alpha}-\beta)\|^2 &\le \frac{2}{n}\inf_{v\in[G]}\|X(\beta-v)\|^2  +  C_{\delta,R}^{(1)}\frac{\tau^2}{n}\Cond(D)\|\Lambda\|^2\sum_{i=1}^{K}\Tr(\Sigma^{2i}) \\
    &\quad + C_{\delta,R}^{(2)}\frac{\tau^2}{n}\Cond(D)^{\frac{1}{2}}\|\Lambda\|\bigg(\sum_{j=1}^{K}\Tr(\Sigma^{i})^{2}\bigg)^{\frac{1}{2}} \\
    &\quad+ \Big(C_{\delta,R}^{(3)} +4(\rho_{\min}(R)^{-1}+\rho_{\min}(R)^{-2})\Big)\sum_{j=1}^{K}\alpha_j\Lambda_j^2.
\end{align*}
The proof can be concluded with the expression of $\alpha$ given in \eqref{eq:alphai}, with $D_{\delta,R}'$ such that
$$D_{\delta,R}'\geq\max\big(C_{\delta,R}^{(1)},C_{\delta,R}^{(2)},C_{\delta,R}^{(3)}+4(\rho_{\min}(R)^{-1}+\rho_{\min}(R)^{-2})\big).$$ Note that $\rho(R)\geq \Tr(R)/K=1$ and $\Cond(R)\geq 1$. Hence, it is straightforward that it is possible to consider $D_{\delta,R}'=c_\delta'\Cond(R)^4$ with well chosen $c_\delta'$.

\subsection{A more precise result for Ridge PLS estimator}
As the bound displayed in \cref{Th:1}, \cref{Th:2} has been simplified for the sake of clarity. We give a more precise result below.

\begin{theorem}\label[thm]{Th:4} 
Let $\delta \in (0,1)$. Suppose that \ref{ass:inverse} holds, and set 
$$\alpha_{i}=c_{\delta}K\frac{\tau^2}{n}\rho(\Sigma)^{i}\Tr(\Sigma^{i})\quad \forall 
i \in \lbrace1,...,K\rbrace,$$
Then, with a probability larger than $1-\delta$,
\begin{align*}
\frac{1}{n}\|X\hat{\beta}_{K,\alpha}-X\beta\|^2
& \le\frac{2}{n}\underset{v\in[G]}{\inf}\|X(\beta-v)\|^2
    +\mathcal{C}^{'(1)}_{\delta,R} \frac{\tau^2}{n} \left(\sum_{i=1}^{K}\sqrt{\Tr(\Sigma^{2i})}\lvert \Lambda_i\rvert\right)^2\\
& \quad +  \mathcal{C}^{'(2)}_{\delta,R}\|\tilde{\Lambda}\|^2 \frac{\tau^2}{n}\sum_{j=1}^{K}\frac{\rho(\Sigma)^{2j}}{\sigma^T\Sigma^{2j-1}\sigma}\\
 & \quad + \mathcal{C}^{'(3)}_{\delta,R} \frac{\tau^2}{n} \left(\frac{1}{K c_{\delta}}\sum_{i=1}^{K}\bar\Lambda_i\Tr(\Sigma^{i})+\sum_{i=1}^{K}\bar{\Lambda}_{i}\rho(\Sigma)^{i}\right)\\
 &  \quad + \C_{\delta,R}^{'(4)}\frac{\tau^2}{n}Kc_{\delta}\sum_{i=1}^{K}\bigg(\rho(\Sigma)^{i}\Tr(\Sigma^{i})\bigg)\Lambda_i^2.
 \end{align*}
 for some constants $\mathcal{C}^{'(j)}_{\delta,R}$, $j\in \lbrace 1,2,3,4 \rbrace$ depending only from $\delta$ and $R$.
 \end{theorem}

\begin{proof}
The results follows from \eqref{Eq:BiaisDecomp} and \eqref{Eq:RidgeBound}.
\end{proof}

%%%%%%%%%%%%%%%%%%%%%%%%%%%%%%%%%%%%%%%%%%%%%%%%%%%%%%%%%%%%%%%%%%%%%%%%%%%%%%%%%%%%%%%%%%%%%%

%\bibliographystyle{abbrvnat}
%\bibliography{biblio}

\printbibliography

\end{document}